\documentclass[a4,10pt]{amsart}
\usepackage[T1]{fontenc}
\usepackage[utf8]{inputenc}
\usepackage{amsmath}
\usepackage{amsthm}
\usepackage{amssymb}
\usepackage{layout}
\theoremstyle{plain}
\newtheorem{theorem}{Theorem}
\newtheorem{lemma}{Lemma}
\newtheorem{prop}{Proposition}
\theoremstyle{remark}
\newtheorem{remark}{Remark}
\newtheorem{example}{Example}
\newcommand{\defcolon}{\;\raisebox{.075ex}{:}\!\!=}
\title{Local Type III metrics with holonomy in $\mathrm{G}_2^*$}
\author{Christian Volkhausen}
\address{Christian Volkhausen, Institut für Mathematik und Informatik, Universität Greifs\-wald, Walther-Rathenau-Straße 47, 17489 Greifswald, Germany}
\curraddr[]{Max-Planck-Institut für Plasmaphysik, Wendelsteinstraße 1, 17491 Greifs\-wald, Germany}
\email{christian.volkhausen@ipp.mpg.de}
\thanks{I am very grateful to Ines Kath for introducing me into the field of holonomy theory as well as her support and useful advises on this article.}
\keywords{Holonomy, pseudo-Riemannian manifold, exterior differential systems, exceptional Lie group}
\subjclass[2000]{Primary 53C29; Secondary  53C50, 53C10}
\begin{document}
\maketitle
\begin{abstract}
Fino and Kath determined all possible holonomy groups of seven-dimensional pseu\-do-Rie\-man\-nian manifolds contained in the exceptional, non-compact, simple Lie group $\mathrm{G}_2^*$ via the corresponding Lie algebras. They are distinguished by the dimension of their maximal semi-simple subrepresentation on the tangent space, the socle. An algebra is called of Type I, II or III if the socle has dimension 1, 2 or 3 respectively. This article proves that each possible holonomy group of Type III can indeed be realized by a metric of signature $(4,3)$.  For this purpose, metrics are explicitly constructed, using Cartan's methods of exterior differential systems, such that the holonomy of the manifold has the desired properties.

\end{abstract}

\section{Introduction}
\label{intro}
Holonomy theory is a very rich field of research between differential geometry and algebra for about a century. When Cartan classified Riemannian symmetric spaces in 1926 \cite{Cartan1926}, he thereby also classified the irreducible holonomy groups of such spaces. Irreducible here means that the natural representation of the holonomy group on the tangent space, the \textit{holonomy representation}, does not leave invariant any proper subspaces. Thirty years later Berger classified simply-connected semi-Riemannian manifolds with irreducible holonomy which are not locally symmetric \cite{Berger1955}. The result is the well known Berger list for Riemannian and pseudo-Riemannian manifolds. In 1987, Bryant finally showed that manifolds with holonomy $\mathrm{G}_2$ and $\mathrm{Spin}(7)$ exist\cite{Bryant1987}. Thus, in the Riemannian case each group on Berger's list is indeed realizable by a metric, i.e. there exists a Riemannian metric such that its holonomy group is equal to that one on Berger's list. It follows that the (reduced) holonomy groups of irreducible, not locally symmetric Riemannian manifolds are completely classified.

In the case of pseudo-Riemannian manifolds the situation is more complicate. Since the holonomy representation can have degenerate subspaces, it is not sufficient to classify only irreducible holonomy groups. It is necessary to consider \textit{indecomposable} ones as well, where indecomposability means that the holonomy representation does not leave invariant any proper, non-degenerate subspace.

While the holonomy of irreducible pseudo-Riemannian manifolds is also classified by Berger's list, little is known about the holonomy of the indecomposable ones, especially about those with index greater than 2. Holonomy groups of Lorentzian manifolds were classified by Leistner \cite{Leistner2007,Galaev2016}, while Galaev classified those of pseudo-K\"ahlerian manifolds of index 2 \cite{Galaev2016a} and Einstein not Ricci-flat pseudo-Riemannian manifolds \cite{Galaev2018}.  Beside these classifications only some results associated with certain special geometries are known (cf. \cite{Graham2012, Leistner2012, Fino2015}). 

 This article treats of pseudo-Riemannian manifolds with signature $(4,3)$ and their ho\-lo\-no\-my.  Throughout this article, manifolds are assumed to be simply-connected such that holonomy coincide with reduced holonomy and Berger's results are applicable. If a manifold with signature $(4,3)$ has irreducible holonomy, then it is either generic $\mathrm{SO}(4,3)$ or $\mathrm{G}_2^{*}$ by Berger's list. Therefore, it is natural to ask which proper subgroups of $\mathrm{G}_2^{*}$ are holonomy groups, i.e. have indecomposable holonomy representation. As the manifold is simply-connected, the holonomy group is connected and one may equivalently asks which subalgebras of $\mathfrak{g}_2^{*}$ are holonomy algebras, i.e. the Lie algebra of a holonomy group. Here, $\mathfrak{g}_2^{*}$ denotes the Lie algebra of $\mathrm{G}_2^{*}$. There are already some results in this direction, e.g. in \cite{Kath1998} indefinite symmetric spaces were classified, and it turns out that their holonomy algebras are two or three dimensional and abelian. 

The article is based upon results of Fino and Kath in \cite{Fino2016} and is part of a project whose aim is to classify and realize all proper subgroups of $\mathrm{G}_2^{*}$ which are holonomy groups. As said before, the classification is due to the holonomy algebras  $\mathfrak{h}\subset \mathfrak{g}_2^{*} \subset \mathfrak{so}(4,3)$ and up to conjugation in $\mathrm{SO}(4,3)$. Algebraic restrictions on the holonomy algebras arise from indecomposability, i.e. the natural representation of the holonomy algebra $\mathfrak{h}\subset \mathfrak{g}_2^{*} \subset \mathfrak{so}(4,3)$ on the tangent space $T_xM \cong \mathbb{R}^{4,3}$ has to leave invariant an isotropic subspace if $\mathfrak{h}\neq \mathfrak{g}_{2}^{*}$, and the fact that each holonomy algebra has to be a Berger algebra, i.e. satisfies Berger's first criterion. Possible holonomy algebras are further divided into three different types by their maximal semi-simple subrepresentation, the so called \textit{socle}. The algebras are named as of Type I, II or III if the socle has dimension 1, 2 or 3 respectively.  In \cite{Fino2017} it is shown by Fino and Kath that all possible holonomy algebras of Type I can indeed be realised, i.e. for each algebra there exists a metric $g$ with signature $(4,3)$ such that the pseudo-Riemannian manifold $(M,g)$ has holonomy according to the particular algebra. In the present article these investigations are continued and it is shown that the same holds for all algebras of Type III as well. Hence, metrics $g$ will be explicitly constructed such that each Type III algebra is the holonomy algebra of a pseudo-Riemannian manifold $(M,g)$ with signature $(4,3)$.

In the following section a brief review of the characterisation of $\mathrm{G}_2^{*}$ and the algebras of Type III is given. Also,  most of the notation is introduced and some results from \cite{Fino2016}, which are necessary to state the classification theorem for Type III algebras, are provided. Section 3 outlines the construction approach for local metrics with specific holonomy properties using the concept of $H$-structures and Cartan's methods of exterior differential systems. Afterwards, the `machinery' developed in this way is used to prove the main theorem of this article:
\begin{theorem}\label{Thm:Main}
Each possible holonomy algebra of Type III can be realised by a pseudo-Rie\-man\-nian metric of signature $(4,3)$.
\end{theorem}
As the proof is constructive, also an example is provided in each case considered. Furthermore, the construction yields (quasi-)normal forms for the metrics. Here, `quasi' refers to some additional conditions imposed on the metrics by certain PDEs. These (quasi-)normal forms can be used to describe the deformation spaces of the metrics. This aspect is not further investigated, as it is beyond the scope of this article. 
%
%
\section{Type III holonomy algebras}
\label{sec:1}
The group $\mathrm{G}_2^{*}$ is strongly related to the exceptional, complex, simple Lie group $\mathrm{G}_2$, which is well known from classical Lie theory. The complex  Lie algebra $\mathfrak{g}_2$ of $\mathrm{G}_2$ has two real forms: $\mathfrak{g}_2^{c}$ and $\mathfrak{g}_2^{*}$, corresponding to the compact group $\mathrm{G}_2^{c}$ and the non-compact $\mathrm{G}_2^{*}$. Each of these groups $\mathrm{G}_2,\,\mathrm{G}_2^c,\,\mathrm{G}_2^{*}$ appears on Berger's list, either for Riemannian or pseudo-Riemannian manifolds \cite{Berger1955,Bryant1987}. 

There are several ways to characterize the group $\mathrm{G}_2^{*}$. Here, an approach via an alternating 3-form is used, which traces back to Engel when he first described the complex group $\mathrm{G}_2$ \cite{Engel1900}. He found that $\mathrm{G}_2$ is the stabiliser of a generic complex 3-form $\omega^{\mathbb{C}}$, i.e. $\omega^{\mathbb{C}}$ has open $\mathrm{GL}(7,\mathbb{C})$-orbit. Passing to the real numbers this orbit splits into two, each of them corresponding to real, generic, $\mathrm{GL}(7,\mathbb{R})$-equivalent 3-forms \cite{Reichel1907, Agricola2008}. One can show that the stabilisers of these forms are exactly $\mathrm{G}_2^{*}$ and $\mathrm{G}_2^{c}$. For any real vector space $V$ one can thus describe $\mathrm{G}_2^{*}$ as the stabiliser of a generic real 3-form $\omega \in \Lambda^{\!3}V^{*}$ on $V$.

 Before elaborating the details, notice some other characterisations based on stabiliser properties of $\mathrm{G}_2^{*}$. It is  also the stabiliser of a non-isotropic element of the real spinor representation of $\mathrm{Spin}(4,3)$ or of an specific cross product on $\mathbb{R}^{4,3}$. Each of these approaches correspond to each other, see \cite{Kath1998} for a detailed description. Furthermore, $\mathrm{G}_2^{*}$ is the automorphism group of the split octonion algebra, which is a more classical characterisation approach.

Now, let $V$ be a 7-dimensional real vector space on which $\mathrm{G}_{2}^{*}$ act as a subgroup of $\mathrm{GL}(V)$. 
 Identify $V \cong \mathbb{R}^7$ and choose the canonical basis as the basis $b_1,\ldots,b_7$ for $V$. Then, $\mathrm{G}_2^{*}$ can be considered as the stabiliser of the real 3-form $\omega$ given by
\begin{align*}
\omega \equiv \sqrt{2} (b^{167} + b^{235}) - b^4 \wedge (b^{15}+b^{26}+b^{37}) \hspace{.5em},
\end{align*}
where $b^1,\ldots,b^7$ denotes the dual basis and $b^{ijk} = b^i \wedge b^j \wedge b^k$. This 3-form induces an inner product on $V$  by 
\begin{align}
 \langle X,Y\rangle &= \frac{1}{6} (\iota_{X}(\omega)) \wedge (\iota_{Y}(\omega) \wedge \omega) \in \Lambda^{\!7}V \cong \Lambda^{\!7} \mathbb{R}^{7} \cong \mathbb{R}\nonumber\\
&= \left(2(b^1\cdot b^5 + b^2 \cdot b^6 + b^3\cdot b^7) - (b^4)^2\right)(X,Y)\label{Eq:Metric} \hspace{.5em},
\end{align}
where $\iota$ is the interior product. See \cite{Bryant1987} for details. With respect to this inner product consider $\mathrm{G}_2^{*}$ as a subgroup of $\mathrm{SO}(4,3)$, now under the identification $V \cong \mathbb{R}^{4,3}$. Elements of the Lie algebra $\mathfrak{g}_2^{*}$ are then given by matrices of the form
\begin{align}\label{Eq:MatrixElements}
\begin{pmatrix}
s_1+s_4 & -s_{10} & s_9 & \sqrt{2}s_6 & 0 & -s_{11} & -s_{12} \\
-s_8 & s_1& s_2 & \sqrt{2}s_9  & s_{11}& 0 & s_6 \\
s_7 & s_3 & s_4 & \sqrt{2}s_{10}  &s_{12}& -s_6 & 0 \\
\sqrt{2}s_5 & \sqrt{2}s_7 & \sqrt{2}s_8 & 0 & \sqrt{2} s_6& \sqrt{2}s_9 &\sqrt{2}s_{10} \\
0 & s_{13} & s_{14} & \sqrt{2}s_5 &-s_1-s_4 & s_8 &-s_7 \\
-s_{13} & 0 & -s_5 & \sqrt{2}s_7 & s_{10} &-s_1  & -s_3 \\
-s_{14} & s_5 & 0 & \sqrt{2}s_8 & -s_9 & -s_2 & -s_4  \\
\end{pmatrix}
\end{align}
with $s_i \in \mathbb{R}$. The Lie algebra $\mathfrak{g}_2^{*}$ also exhibit the structure of a $|3|$-graded Lie algebra. Thus, it exists a direct sum vector space decomposition 
\begin{align*}
\mathfrak{g}_{2}^{*} = \mathfrak{g}_{-3} \oplus \mathfrak{g}_{-2} \oplus \mathfrak{g}_{-1} \oplus \mathfrak{g}_0 \oplus \mathfrak{g}_1 \oplus \mathfrak{g}_2 \oplus \mathfrak{g}_3
\end{align*}
such that $\left[\mathfrak{g}_i,\mathfrak{g}_j\right] \subset \mathfrak{g}_{i+j}$. The components $\mathfrak{g}_i$ can simply be computed from the root diagram of $\mathfrak{g}_{2}^{*}$ (cf. \cite{Fulton2004}), e.g. $\mathfrak{g}_0$ corresponds to the subalgebra generated by the longer primitive root vector, its negative and the (canonical) Cartan subalgebra. Since $\mathfrak{g}_0 \oplus \ldots \oplus \mathfrak{g}_3$ also contains the smaller primitive root vector, it contains a Borel subalgebra of $\mathfrak{g}_2^{*}$.

\begin{remark}
The Lie algebra $\mathfrak{g}_{2}^{*}$ also carries  the structure of a $|2|$-graded  Lie algebra. In this case $\mathfrak{g}_0$ is generated by the shorter primitive root vector, its negative and the Cartan subalgebra. This grading will not be considered here further, but it is the  appropriate one when dealing with Type II algebras.
\end{remark}

With respect to the metric $g$, the line $\mathbb{R}\cdot b_1$ is isotropic. The subalgebra $\mathfrak{p}_1 \subset \mathfrak{g}_2^{*}$, which is the stabiliser of this isotropic line, is given by 
\begin{align*}
\mathfrak{p}_1 = \left\lbrace m \in \mathfrak{g}_2^{*} \,\middle|\; s_i = 0 \; \text{for}\; i \in \{5,7,8,13,14\} \right\rbrace = \mathfrak{g}_0 \oplus \mathfrak{g}_1 \oplus \mathfrak{g}_2 \oplus \mathfrak{g}_3 \hspace{.5em}.
\end{align*}
Note that $\mathfrak{p}_1$ is parabolic. 

The algebraic considerations above can be used to describe manifolds with holonomy contained in $\mathrm{G}_2^{*}$. Let $M$ be smooth manifold with $\dim M = 7$. The vector space $V$ is now taken to be the tangent space at some $x \in M$, i.e. $V = T_xM$. Next, consider the principal frame bundle $\mathrm{GL}(M)$ over $M$ with structure group $\mathrm{GL}(7,\mathbb{R})$. A $\mathrm{G}_{2}^{*}$-\textit{structure} on $M$ is a reduction of $\mathrm{GL}(M)$ along the inclusion $\mathrm{G}_{2}^{*} \hookrightarrow \mathrm{GL}(7,\mathbb{R})$.  By the considerations above, $\mathrm{G}_{2}^{*}$ defines a metric $g$ on $M$ via the inner product. Thus, $\mathrm{G}_{2}^{*} \subset \mathrm{SO}(4,3)$ and the reduction can be thought  along the inclusion $\mathrm{G}_{2}^{*} \hookrightarrow \mathrm{SO}(4,3)\hookrightarrow \mathrm{GL}(7,\mathbb{R})$. In this case, it is known (cf. \cite{Niedzialomski2015}) that the holonomy of $M$ is contained in the structure group $\mathrm{G}_2^{*}$ if and only if the $\mathrm{G}_{2}^{*}$-structure is torsion-free.
 
Since the objects of investigation are holonomy groups which are proper subgroups of $\mathrm{G}_2^{*}$, or equivalently subalgebras $\mathfrak{h}\subset \mathfrak{g}_{2}^{*}$, the holonomy representation has to be indecomposable but not irreducible (which would implie $\mathfrak{h} = \mathfrak{g}_{2}^{*}$). Thus, there is at least one $\mathfrak{h}$-invariant isotropic subspace of $V$. Using this fact, it was proven in \cite{Fino2016} that Type III algebras are in general of the form
\begin{align*}
\mathfrak{h}^{\mbox{\scriptsize \textit{III}}} = \left\lbrace h(A,y,v) \;\middle|\; A \in \mathfrak{gl}(2, \mathbb{R}), \; v \in \mathbb{R}, \; y \in \mathbb{R}^{2} \right\rbrace = \mathfrak{g}_{0} \oplus \mathfrak{g}_2 \oplus \mathfrak{g}_{3} \subset \mathfrak{p}_1 \hspace{.5em},
\end{align*}
where
\begin{align*}
h(A,v,y) =\begin{pmatrix}
\mathrm{tr}\; A & 0 & 0 & \sqrt{2}v & 0 & -y_{1} & -y_{2} \\
0 & a_1& a_2 & 0  & y_{1}& 0 & v \\
0 & a_3 & a_4 & 0  &y_{2}& -v & 0 \\
0 & 0 & 0 & 0 & \sqrt{2} v& 0 &0 \\
0 & 0 & 0 & 0 &-\mathrm{tr}\; A & 0 &0 \\
0 & 0 & 0 & 0 &0 &-a_1  & -a_3 \\
0 & 0 & 0 & 0 & 0 & -a_2 & -a_4  \\
\end{pmatrix},
\end{align*}
with elements $
A = \begin{pmatrix} a_1 & a_2 \\ a_3 & a_4\end{pmatrix} \in \mathfrak{gl}(2, \mathbb{R}), \; y = (y_1, y_2)^{\top} \in \mathbb{R}^2  , \; v \in \mathbb{R}
$. The commutator of two elements $h(A,v_A,y_A)$ and $h(B,v_B,y_B)$ is given by 
\begin{align}\label{Eq:Commutator}
[h(A,v_A,y_A),h(B,v_B,y_B)] &=\\\nonumber
 &\hspace{-2.5cm}h\left([A,B], \mathrm{tr}(A) \cdot v_B - v_A \cdot \mathrm{tr}(B), (A+\mathrm{tr}(A))\cdot y_B - (B+\mathrm{tr}(B))\cdot y_A\right)\hspace{.5em}.
\end{align}

It is appropriate to introduce some further notation at this point, which is adapted from \cite{Fino2016}. First, define the Lie algebras
\begin{align*}
\mathfrak{m}(1,2) &\defcolon  \left\lbrace h(0,v,y) \;\middle|\; v \in \mathbb{R}, y = (y_1,y_2)^{\top} \in \mathbb{R}^{2} \right\rbrace = \mathfrak{g}_2 \oplus \mathfrak{g}_3\\
\mathfrak{m}(1,1) &\defcolon  \left\lbrace h(0,v,(y_1,0)^{\top}) \;\middle|\; v, y_1 \in \mathbb{R} \right\rbrace \subset \mathfrak{g}_2 \oplus \mathfrak{g}_3 \hspace{.5em}.
\end{align*}
Then, under the identification
\begin{align*}
\mathfrak{gl}(2,\mathbb{R}) \sim \{h(A,0,0) \;|\; A \in \mathfrak{gl}(2,\mathbb{R})\}\hspace{.5em},
\end{align*}
 consider the algebra  $\mathfrak{h}^{\mbox{\scriptsize \textit{III}}}$ as a semi-direct product
\begin{align*}
 \mathfrak{h}^{\mbox{\scriptsize \textit{III}}} = \mathfrak{gl}(2,\mathbb{R}) \ltimes \mathfrak{m}(1,2)\hspace{.5em},
\end{align*}
where elements $A \in \mathfrak{gl}(2,\mathbb{R})$ act on $\mathfrak{m}$ by 
\begin{align*}
A \cdot h(0,v,y) = h(0, (\mathrm{tr}\;A) \cdot v, (A + \mathrm{tr}\; A ) \cdot y) \hspace{.5em}.
\end{align*}
Furthermore, identify $\mathrm{GL}(2,\mathbb{R})$ with a subgroup of diagonal block-matrices in $\mathrm{G}_2^{*}$ using the map
\begin{align*}
\mathrm{GL}(2,\mathbb{R}) \ni a \mapsto \mathrm{diag}\left(\det a, a,1, (\det a)^{-1}, (a^{\top})^{-1}\right) \subset \mathrm{G}_{2}^{*}
\end{align*}
and note the adjoint action on $\mathfrak{h}^{\mbox{\scriptsize \textit{III}}}$ given by
\begin{align*}
\mathrm{Ad}(a)\left(h(A,v,y)\right) = h(aAa^{-1}, \mathrm{det}(a) \cdot v, \mathrm{det}(a)\cdot ay) \hspace{.5em}.
\end{align*}
Later on, elements of $\mathfrak{h}$ are conjugated several times with elements of the form $\exp(h(0,v,y))$. The adjoint action is given by 
\begin{align}
\mathrm{Ad}\left(\exp(h(0,\bar{v},0))\right)\left(h(A,v,y)\right) &= h\left(A,v- (\mathrm{tr}\; A)\cdot \bar{v},y\right) \label{Eq:Conj_v}\\
\mathrm{Ad}\left(\exp(h(0,0,\bar{y}))\right)\left(h(A,v,y)\right) &= h\left(A,v,y- (A + \mathrm{tr}\; A )\cdot \bar{y}\right) \hspace{.5em}.\label{Eq:Conj_y}
\end{align}
The classification theorem for Type III algebras was already proven in \cite{Fino2017}, but it is restated here as it becomes more simple when dealing with Type III algebras alone. First of all, it is a well known fact that every holonomy algebra is a Berger algebra, i.e. $\mathfrak{h}\subset \mathfrak{h}^{\mbox{\scriptsize \textit{III}}}$ has to satisfy Berger's first criterion. Next, consider 
\begin{align*}
\mathcal{K}\!(\mathfrak{h}) = \left\lbrace R \in \Lambda^{\!2}V^{*} \otimes \mathfrak{h} \;\middle|\; \forall X,Y,Z \in V : R(X,Y)Z+R(Y,Z)X+R(Z,X)Y =0 \right\rbrace \hspace{.05em},
\end{align*}
the space of algebraic curvature endomorphisms, and define the space
\begin{align*}
\underline{\mathfrak{h}} \defcolon \mathrm{span} \left\lbrace R(X,Y) \;\middle|\; X,Y \in V , R \in \mathcal{K}(\mathfrak{h}) \right\rbrace \hspace{.5em}.
\end{align*}
Berger's first criterion now implies $\mathfrak{h} = \underline{\mathfrak{h}}$ for $\mathfrak{h}$ being a holonomy algebra. A parametrisation for the space $\mathcal{K}\!(\mathfrak{h})$ is given in the case of Type I algebras by \cite{Fino2016}. Since $ \mathfrak{h}^{\mbox{\scriptsize \textit{III}}} \subset  \mathfrak{h}^{\mbox{\scriptsize \textit{I}}}$, in principle the same parametrisation is used for Type III algebras. 
\begin{lemma}
In the case of Type III algebras the space $\mathcal{K}(\mathfrak{h})$ can be parametrised by $a_1,a_2$, $a_3,d_1$, $d_2 \in \mathbb{R}$, $c_k,j_k \in \mathbb{R}$ with $k= 1,\ldots,4$ and $v_1,v_2,t \in \mathbb{R}$, where $R = h(A,v,y)$ is given by the data in Table \ref{Tab:RTypIII}.
\end{lemma}
\begin{proof}
The Lemma follows directly from Proposition 2.8 in \cite{Fino2016} and the fact that $\mathfrak{h}^{\mbox{\scriptsize \textit{III}}} \cap \mathfrak{g}_1 = \emptyset$.
\end{proof}
Some last piece of notation is necessary before stating the main theorem from \cite{Fino2016} for Type III holonomy algebras. Therefore, define the following matrix algebras
\begin{align*}
\mathfrak{co}(2) &= \left\lbrace \begin{pmatrix} a & -b \\ b & a  \end{pmatrix} \; \middle| \; a, b \in \mathbb{R} \right\rbrace \quad
\mathfrak{d} = \left\lbrace \begin{pmatrix} a & 0 \\ 0 & d  \end{pmatrix} \; \middle| \; a, d \in \mathbb{R} \right\rbrace \\
&\hspace{1cm}\mathbb{R} \cdot \mathrm{diag}(1,\mu) =  \left\lbrace c\cdot \begin{pmatrix}1& 0 \\ 0 & \mu  \end{pmatrix} \; \middle| \; c \in \mathbb{R} \right\rbrace,\; \mu \in \mathbb{R}\\
&\hspace{1cm}\mathbb{R}\cdot C_a = \left\lbrace  c \cdot \begin{pmatrix} a & -1 \\ 1 & a  \end{pmatrix} \; \middle| \;  c \in \mathbb{R} \right\rbrace, \; a \in \mathbb{R}\hspace{.5em}.
\end{align*}
\begin{theorem}[\cite{Fino2015}, Type \makebox{III} holonomy algebras]\label{Th:ClassTheorem}
If $\mathfrak{h}$ is of Type III, then there exists a basis of $V$ such that we are in one of the following cases
\begin{itemize}
\item[1.] $\mathfrak{h} = \mathfrak{a} \ltimes \mathfrak{m}(1,2)$ with $\mathfrak{a} \in \{\mathfrak{gl}(2,\mathbb{R}),\mathfrak{sl}(2,\mathbb{R}),\mathfrak{co}(2),\mathfrak{d}\}$
\item[2.] $\mathfrak{h} = \mathfrak{a} \ltimes \mathfrak{m}(1,k)$ with $\mathfrak{a} \in \{\mathbb{R}\cdot \mathrm{diag}(1,0),0\}$, where $k = 1,2$.
\end{itemize}
\end{theorem}
\begin{proof}
Let $b_i$ be the basis chosen as in (\ref{Eq:MatrixElements}), with dual basis $b^i$ as in (\ref{Eq:Metric}). By \cite{Fino2016} the three-\-di\-men\-sio\-nal socle $S$ of $\mathfrak{h}$ is spanned by $b_1,b_2,b_3$, where the line spanned by $\mathbb{R}\cdot b_1$ is isotropic. The socle splits under the semi-simple action of $\mathfrak{h}$ into $S = \mathbb{R}\cdot b_1 \oplus \hat{S}$, where $\hat{S} = \mathrm{span}\{b_2,b_3\}$ (cf. \cite{Fino2016}). Since $\mathfrak{h}$ acts semi-simply on $\hat{S}$, $\mathfrak{a}$ acts semi-simply on $\mathbb{R}^{2}$. Thus, $\mathfrak{a} \in \{0,\mathfrak{gl}(2,\mathbb{R}),\mathfrak{sl}(2,\mathbb{R}\})$ or it is conjugated to one of the algebras $\mathfrak{co}(2),\mathfrak{d},\mathbb{R}\cdot \mathrm{diag}(1,\mu)$ or $\mathbb{R}\cdot C_a$. 

First, exclude $\mathfrak{a}=\mathbb{R}\cdot C_a$. By Table \ref{Tab:RTypIII} this implies the system of linear equations
\begin{align*}
d_1 = c_2 = -c_3 \qquad -d_2 = c_1 = c_4 \qquad d_1 = \alpha \cdot c_1 \qquad c_1 = \alpha \cdot c_3
\end{align*}
for an arbitrary parameter $\alpha \in \mathbb{R}$. The system has only the trivial solution, which is a contradiction to  Berger's first criterion $h(C_a,v_0,y_0) \in \underline{\mathfrak{h}}$, for some $v_0 \in \mathbb{R}$ and $y_0 \in \mathbb{R}^{2}$.

Next, note that in the remaining cases $\mathfrak{v} = \left\lbrace v \in \mathbb{R} \;\middle|\; h(0,v,y) \in \mathfrak{h}, y\in \mathbb{R}^2 \right\rbrace \neq 0$. In order to see this assume $\mathfrak{v} = 0$ and consider first the case $\mathfrak{a} = \mathfrak{sl}(2,\mathbb{R})$. By indecomposability of the holonomy representation the kernel of $\mathfrak{h}$ is necessarily isotropic. Then, for each $A \in \mathfrak{a}$ there exists a unique element $v_a$ such that $h(A,v_a,y)\in \mathfrak{h}$ for some $y \in \mathbb{R}^{2}$. Otherwise the non-isotropic vector $b_4$ would be in $\ker \mathfrak{h}$. The map $\omega: A \mapsto v_A$ is a cocycle with respect to the representation $\rho: \mathfrak{a} \to \mathfrak{gl}(\mathbb{R}) \cong \mathbb{R}$, $\rho(A)v = \mathrm{tr}(A)\cdot v$, which is just the trivial representation in the case of $\mathfrak{sl}(2,\mathbb{R})$. This can be seen easily seen by calculating $v_{\left[A,B\right]}  = \mathrm{tr}(A) \cdot v_B - v_A \cdot \mathrm{tr}(B)$ via Eq. (\ref{Eq:Commutator}).

By semi-simplicity of $\mathfrak{sl}(2,\mathbb{R})$ and Whitehead's first lemma $\omega$ is a coboundary such that $v_A = \mathrm{tr}(A)\cdot\hat{v}$ for some $\hat{v} \in \mathbb{R}$. Now, conjugating with the element $\exp\!\left(h(0,\hat{v},0)\right)$ according to \eqref{Eq:Conj_v}, the projection to $\left\lbrace h(0,v,0) \;\middle|\; v \in \mathbb{R} \right\rbrace$ is trivial. This implies $b_4 \in \ker\mathfrak{h}$ which is a contradiction to $\ker \mathfrak{h}$ being necessarily isotropic. Hence, $\mathfrak{v} \neq 0$.

For the other algebras one can argue similarly. Again, assume $\mathfrak{v} = 0$ and by indecomposability $\ker\mathfrak{h}$ is isotropic which implies the existence of a unique $v_A$ for each $A\in \mathfrak{a}$ such that $h(A,v_a,y)\in \mathfrak{h}$ for some $y \in \mathbb{R}^{2}$. The map $\omega: A \mapsto v_A$ is still a cocycle. Now, there are two subcases: Firstly, for the unit $I\in \mathfrak{a}$, by above considerations, $0 = v_{[I,A]} = I(v_A) - A(v_I)$ implies $v_A = 2\,A \cdot v_I$. After conjugation with $\exp(h(0,2\cdot v_I,0)$ the projection to $\left\lbrace h(0,v,0) \;\middle|\; v \in \mathbb{R} \right\rbrace$ is again trivial and the contradiction follows as before. Secondly, for $\mathfrak{a} = \mathbb{R}\cdot \mathrm{diag}(1,\mu)$ choose $I' = \mathrm{diag}(1,\mu)$. Then, the same argument holds using $I'$ instead of the unit $I$. 
For $\mathfrak{a} = 0$ the proposition follows directly from indecomposability. 

Next, by $\mathfrak{v}\neq 0$ and Table \ref{Tab:RTypIII} obtain $v_1 \neq 0$ or $v_2 \neq 0$, hence one find that $\mathfrak{y} = \left\lbrace y \in \mathbb{R}^2 \;\middle|\; h(0,0,y) \in \mathfrak{h}\right\rbrace \neq 0$.

 Now, if $\mathfrak{a} \in \{0,\mathfrak{gl}(2,\mathbb{R}),\mathfrak{sl}(2,\mathbb{R}),\mathfrak{co}(2)\}$ one has $\mathfrak{y}=\mathbb{R}^{2}$, as there is no one-dimensional, invariant subspace of $\mathbb{R}^{2}$. Thus, $\mathfrak{a}\subset \mathfrak{h}$ and $\mathfrak{h} \cap \mathfrak{m} = \mathfrak{m}(1,2)$, which implies $\mathfrak{h} = \mathfrak{a} \ltimes \mathfrak{m}(1,2)$.
 
 Assuming $\mathfrak{a} \subset \mathfrak{d}$, it follows that the parameters in column A of  Table \ref{Tab:RTypIII} are identically zero, except $d_1,c_4$. From Table \ref{Tab:RTypIII} immediately follows in the case $\mathfrak{a} = \mathfrak{d}$ that $\mathfrak{y} = \mathbb{R}^{2}$ and subsequently $\mathfrak{h} = \mathfrak{d} \ltimes \mathfrak{m}(1,2)$. If $\mathfrak{a} = \mathbb{R} \cdot \mathrm{diag}(1,\mu) \subset \mathfrak{d}$, Table \ref{Tab:RTypIII} implies $\mu = 0$ and $d_1 \neq 0$ such that, possibly after conjugation, $\mathfrak{h} \cap \mathfrak{m} = \mathfrak{m}(1,k)$, $k = 1,2$ and consequently $\mathfrak{h} = \mathbb{R} \cdot \mathrm{diag}(1,0) \ltimes \mathfrak{m}(1,k)$.
 
 Finally, in the case $\mathfrak{a} = 0$ follows  $\mathfrak{y} = \mathbb{R}^{2}$ or $\mathfrak{y} = (1,0)^{\top}$ by Table \ref{Tab:RTypIII} after suitable conjugation. 
\end{proof}
\begin{table}[!h]
\begin{center}
\caption{Parametrisation of curvature endomorphisms which are able to span the algebras of Type III. Based on the table for Type I algebras in \cite{Fino2016}.}\label{Tab:RTypIII}
\begin{tabular}{|c|c|c|c|}
\hline &&& \\[-2ex]
$R_{ij}=R\!\left(b_i,b_j \right)$ & $A$ & $v$ & $y$ \\[.75ex]
\hline &&&\\[-1.5ex]
$R_{15}$ & $0$ & $0$  & $(d_1+c_2,c_1+c_4)$ \\
$R_{26}$ &$\begin{pmatrix} -a_1 & -a_2 \\ -a_3 & a_1\\ \end{pmatrix}$ & $0$  & $(d_1,c_1)$ \\[3ex]
$R_{27}$ & $\begin{pmatrix} -a_3 & a_1 \\ j_1 & a_3\\ \end{pmatrix}$ & $0$  & $(c_2,c_3)$ \\[3ex]
$R_{36} $ & $\begin{pmatrix} -a_2 &j_2 \\ a_1 & a_2\\ \end{pmatrix}$ & $0$  & $(d_2,c_2)$ \\[3ex]
$-R_{67} = \frac{1}{\sqrt{2}}R_{45}$ &$0$& $0$  & $(v_1,v_2)$ \\[2ex]
$R_{56}$  &$\begin{pmatrix}d_1 & d_2 \\ c_1 & c_2\\ \end{pmatrix}$ & $v_1$  & $(j_3,t)$ \\[3ex]
$R_{57}$  &$\begin{pmatrix} c_1 & c_2 \\c_3 & c_4\\ \end{pmatrix}$ & $v_2$  & $(t,j_4)$ \\[3ex]
\hline
\multicolumn{4}{|c|}{\rule{0pt}{2.0ex}$R_{12}=R_{13}=R_{14}=R_{16}=R_{17}=R_{23}=R_{24}=R_{34}=0$}\\[1ex]
\multicolumn{4}{|c|}{$R_{37}=R_{15}-R_{26}$}\\
\hline
\end{tabular}
\end{center}
\end{table}
%
%
\section{Local Type III metrics}\label{sec:2}
In this section a procedure is developed in order to construct (quasi-)normal forms for the metric induced by the $\mathrm{G}_{2}^{*}$-invariant 3-form $\omega$ (cf. (\ref{Eq:Metric})). Then, these (quasi-)normal forms are used to find metrics which realize each of the algebras in Theorem \ref{Th:ClassTheorem}. This will prove Theorem \ref{Thm:Main}. The appraoch is fairly the same used by Fino and Kath for Type I algebras \cite{Fino2017}. 

Assume $M$ has holonomy $H \subset \mathrm{G}_{2}^{*}$. Then, $M$ admits a torsion-free $H$-structure. Holonomy is here considered with respect to the metric $g$ and the Levi-Civita connection $\nabla$ induced by the $\mathrm{G}_{2}^{*}$-invariant (hence $H$-invariant) 3-form $\omega$ as described in Section \ref{sec:1}. Note, every torsion-free $H$-structure is 1-flat \cite{Bryant1987} and Cartan's first structure equation is given by
\begin{align*}
\mathrm{d}\eta = -\theta \wedge \eta \hspace*{.5em},
\end{align*}
where $\eta: TP \to \mathbb{R}^{n}$ and $\theta: TP \to \mathfrak{h}$ is the connection form associated with the Levi-Civita connection. Here, the principal bundle $P$ is the reduction of the frame bundle along $H \hookrightarrow \mathrm{G}_{2}^{*}\hookrightarrow \mathrm{GL}(7,\mathbb{R})$.  Choosing a local frame $(b_i)$ in $P$ the components of the connection form are given by $\theta_{j}^{i} = b^{i}(\nabla b_j)$, where $(b^j)$ denotes the dual frame. For any tangent vector $X$ one has $b^{i}(\nabla b_j)(X)$ $=b^{i}(\nabla_{X} b_j)$.
The construction `machinery' now consists of three steps. First, choose a local frame $(b_i)$ and write down the structure equations as an exterior differential system, where the connection form takes values in the holonomy algebra $\mathfrak{h}$. Choosing local coordinates $x_i$ on $M$, the frame is expressible in terms of local functions using exterior differential calculus methods, i.e. Frobenius theorem and suitable transformations of the local frame by elements of $H$. As the second step, use this adapted frame and calculate the components of the connection form, which reformulates as a PDE system for the local functions on $M$. With respect to this system lastly use Ambrose-Singer holonomy theorem \cite{Ambrose1953} and  calculate the components of the curvature tensor (and their covariant derivatives, if necessary; cf. \cite{Ozeki1956}) such that they span the holonomy algebra $\mathfrak{h}$.

In the relevant case of manifolds $(M,g)$ with holonomy contained in $\mathrm{G}_2^{*}$ the connection form takes values in the algebras of Theorem \ref{Th:ClassTheorem}. Applying the approach described above, main Theorem \ref{Thm:Main} is proven in the following subsections by calculating a metric for each of the 8 cases. 

\begin{remark}
Below bold letters denote differential 1-forms. Since the procedure is quite the same in each case, detailed calculations are provided in the first and most general one while omitting some details of these somewhat laborious calculations later on. Furthermore, note that all metrics finally turn out to be real analytic. This guarantees that the infinitesimal holonomy calculated in the following subsections coincide with the holonomy. 
\end{remark}
\subsection{Holonomy of Type III 1a}
\begin{prop}\label{Th:Case1a}
The holonomy of $(M^{4,3},g)$ is contained in $\mathrm{G}_2^{*}$ and is of Type III 1a if and only if there are local coordinates $x_1,\ldots,x_7$ such that $g = 2(b^1\cdot b^5 + b^2 \cdot b^6 + b^3\cdot b^7) - (b^4)^2$, where
\begin{align*}
b^{1} &= \mathrm{d}{x_1} + r_5(x_4,x_5,x_6,x_7) \cdot  \mathrm{d}{x_5} + r_6(x_1,x_5,x_6,x_7) \cdot  \mathrm{d}{x_6} + r_7(x_1,x_5,x_6,x_7) \cdot  \mathrm{d}{x_7}\\
b^{2} &= \mathrm{d}{x_2} + s_6(x_2,x_3,x_5,x_6,x_7)  \cdot  \mathrm{d}{x_6}  + s_7(x_2,x_3,x_5,x_6,x_7)  \cdot  \mathrm{d}{x_7}\\
b^{3} &= \mathrm{d}{x_3} +t_6(x_2,x_3,x_5,x_6,x_7)  \cdot  \mathrm{d}{x_6}+ t_7(x_2,x_3,x_5,x_6,x_7)   \cdot  \mathrm{d}{x_7}\\
b^{5} & = \mathrm{e}^{u(x_5,x_6,x_7)} \cdot \mathrm{d}{x_5}\\
b^{i} &= \mathrm{d}{x_i} \qquad \qquad i = 4,6,7,
\end{align*}
with functions $r_5,r_6,r_7,s_7$ given by
\begin{align*}
\begin{split}
r_5 &= \frac{x_4}{\sqrt{2}}\left[(\mathrm{e}^{u} \cdot F_6(x_5,x_6,x_7))_{x_7} - (\mathrm{e}^{u} \cdot F_7(x_5,x_6,x_7))_{x_6} - (F_4(x_5,x_6,x_7))_{x_5}\right] \\
&\hspace{9cm}+ F_5(x_5,x_6,x_7)
\end{split}\\
r_6 &= x_1 \cdot (u)_{x_6} + F_6(x_5,x_6,x_7)\\
r_7 &= x_1 \cdot (u)_{x_7} + F_7(x_5,x_6,x_7)\\
s_7 &= t_6 + F_4(x_5,x_6,x_7)
\end{align*}
for arbitrary functions $F_5,F_6,F_7$ and functions $u,s_6,t_6,t_7,F_4$ satisfying the differential equations
\begin{align*}
(u)_{x_6} &= (s_6)_{x_2}+ (t_6)_{x_3}\\
(u)_{x_7} &=(t_7)_{x_3}  + (t_6)_{x_2}\\
  (s_7)_{x_6}+t_7 \cdot (s_6)_{x_3} + s_7 \cdot (s_6)_{x_2}&=  (s_6)_{x_7} + t_6 \cdot (s_7)_{x_3} + s_6 \cdot (s_7)_{x_2}   \\
 (t_7)_{x_6}+ t_7 \cdot (t_6)_{x_3} + s_7 \cdot (t_6)_{x_2}  &=   (t_6)_{x_7} + t_6 \cdot (t_7)_{x_3} + s_6 \cdot (t_7)_{x_2}\hspace{.5em}.
\end{align*}
\end{prop}
\begin{proof}
The structure equation
\begin{align*}
\begin{pmatrix}
\mathrm{d}{b^{1}} \\ \mathrm{d}{b^{2}} \\ \mathrm{d}{b^{3}} \\ \mathrm{d}{b^{4}} \\ \mathrm{d}{b^{5}} \\ \mathrm{d}{b^{6}} \\ \mathrm{d}{b^{7}}
\end{pmatrix} =
-\begin{pmatrix}
\mathbf{a}_1+\mathbf{a}_4 & 0 & 0 & \sqrt{2} \mathbf{v} & 0 & -\mathbf{y}_{1} & -\mathbf{y}_{2} \\
0 &  \mathbf{a}_1 &  \mathbf{a}_2 & 0  & \mathbf{y}_{1}& 0 & \mathbf{v} \\
0 &  \mathbf{a}_3 &  \mathbf{a}_4 & 0  & \mathbf{y}_{2}& - \mathbf{v} & 0 \\
0 & 0 & 0 & 0 & \sqrt{2} \mathbf{v}& 0 &0 \\
0 & 0 & 0 & 0 &- \mathbf{a}_1-\mathbf{a}_4 & 0 &0 \\
0 & 0 & 0 & 0 & 0 &- \mathbf{a}_1  &  -\mathbf{a}_3 \\
0 & 0 & 0 & 0 & 0 &  -\mathbf{a}_2 & -\mathbf{a}_4  \\
\end{pmatrix} \wedge 
\begin{pmatrix}
b^{1} \\ b^{2} \\ b^{3} \\ b^{4} \\ b^{5} \\ b^{6} \\ b^{7}
\end{pmatrix}
\end{align*}
implies that coordinates exist such that $b^6,b^7$ lie in the ideal $\mathcal{I}_0=\mathcal{I}(\mathrm{d}x_6,\mathrm{d}x_7)$ generated by $\mathrm{d}x_6,\mathrm{d}x_7$. Transforming the basis pointwise with an appropriate  element 
\begin{align*}
\mathrm{GL}(2, \mathbb{R})  \ni M = \begin{pmatrix} m_1 & m_2 \\ m_3 & m_4 \end{pmatrix}  \cong  \mathrm{diag}(\det M, M,1, (\det M)^{-1},(M^{\top})^{-1}) \in \mathrm{G}_2^*
\end{align*}
yields $b^6 = \mathrm{d}x_6$ and $b^7 = \mathrm{d}x_7$. This implies $\mathbf{a}_1,\ldots,\mathbf{a}_4 \in \mathcal{I}_0$, since $ \mathrm{d}b^6 = \mathrm{d}\mathrm{d}x_6 = 0$ and $\mathrm{d}b^7 = \mathrm{d}\mathrm{d}x_7 =0$. Furthermore, by $\mathrm{d}b^5 \in \mathcal{I}(b^5)$ the structure equation implies $b^5 = \mathrm{e}^{u} \cdot \mathrm{d}x_5 $ for a suitable local function $u$ and coordinate $x_5$. Since $\mathbf{a}_1,\mathbf{a_4}\in \mathcal{I}_0$, one has $u = u(x_5,x_6,x_7)$. Using the exponential of $u$ here and in comparable cases has practical reasons, as it  slightly simplifies some calculations.

From the structure equation deduce $\mathrm{d}b^4 \in \mathcal{I}(\mathrm{d}x_5)$, the ideal generated by $\mathrm{d}x_5$. Thus, $b^4 = \mathrm{d}x_4 + q\cdot \mathrm{d}x_5$ and after a change of basis by an element $\exp(h(0,v,0))$ follows $b^4 = \mathrm{d}x_4$. Then, $\mathrm{d}b^4 = 0$ implies $\mathbf{v} \in \mathcal{I}(\mathrm{d}x_5)$.

Since $\mathbf{a}_1,\ldots,\mathbf{a}_4  \in \mathcal{I}_0$ and $\mathbf{v} \in \mathcal{I}(\mathrm{d}x_5)$, it follows from the structure equation that $\mathrm{d}b^2$, $\mathrm{d}b^3\in \mathcal{I}_1 = \mathcal{I}(\mathrm{d}x_5,\mathrm{d}x_6,\mathrm{d}x_7)$, the ideal generated by $\mathrm{d}x_5,\mathrm{d}x_6,\mathrm{d}x_7$. This allows the assumption
\begin{align*}
b^2 &= \mathrm{d}x_2 + s_5 \cdot \mathrm{d}x_5 + s_6 \cdot \mathrm{d}x_6 + s_7 \cdot \mathrm{d}x_7 \\ b^3 &= \mathrm{d}x_3 + t_5 \cdot \mathrm{d}x_5 + t_6 \cdot \mathrm{d}x_6 + t_7 \cdot \mathrm{d}x_7 \hspace{.5em}.
\end{align*}
Transforming the basis with an appropriate element $\exp(h(0,0,(y_1,y_2)^{\top}))$ yields
\begin{align*}
b^2 = \mathrm{d}x_2  + s_6 \cdot \mathrm{d}x_6 + s_7 \cdot \mathrm{d}x_7 \qquad b^3 = \mathrm{d}x_3 + t_6 \cdot \mathrm{d}x_6 + t_7 \cdot \mathrm{d}x_7 \hspace{.5em}.
\end{align*}
Then,
\begin{align*}
\mathrm{d}b^2 &= \mathrm{d}s_6 \wedge \mathrm{d}x_6 + \mathrm{d}s_7 \wedge \mathrm{d}x_7 = - \mathbf{a}_1 \wedge b^2 - \mathbf{a}_2 \wedge b^3 - \mathbf{y}_1 \wedge b^5 - \mathbf{v} \wedge b^7 \\
\mathrm{d}b^3 &= \mathrm{d}t_6 \wedge \mathrm{d}x_6 + \mathrm{d}t_7 \wedge \mathrm{d}x_7 = - \mathbf{a}_3 \wedge b^2 - \mathbf{a}_4 \wedge b^3 - \mathbf{y}_2 \wedge b^5 - \mathbf{v} \wedge b^6 
\end{align*}
implies 
\begin{align*}
s_6 &= s_6(x_2,x_3,x_5,x_6,x_7), \qquad s_7 = s_7(x_2,x_3,x_5,x_6,x_7),\\
t_6 &= s_6(x_2,x_3,x_5,x_6,x_7), \qquad t_7 = s_7(x_2,x_3,x_5,x_6,x_7)
\end{align*}
and $\mathbf{y}_1,\mathbf{y}_2 \in \mathcal{I}_1$. Using the latter, $\mathbf{a}_1,\ldots,\mathbf{a}_4  \in \mathcal{I}_0$, $\mathbf{v} \in \mathcal{I}(\mathrm{d}x_5)$ and the equation for $\mathrm{d}b^1$,  by an analogue consideration obtain
\begin{align*}
b^1 = \mathrm{d}x_1 + r_5 \cdot \mathrm{d}x_5  + r_6 \cdot \mathrm{d}x_6 + r_7 \cdot \mathrm{d}x_7
\end{align*}
with functions 
\begin{align*}
r_5 = r_5(x_4,x_5,x_6,x_7), \quad r_6 = r_6(x_1,x_5,x_6,x_7),\quad r_7 = r_7(x_1,x_5,x_6,x_7) \hspace{.5em}.
\end{align*}
Now, with respect to the local coordinates and the metric defined above, the holonomy is contained in $\mathrm{G}_2^{*}$ and of Type III 1a if and only if the following relations between the components of the local connection form are satisfied:
\begin{align*}
b^1(\nabla b_1) = b^2 (\nabla b_2) + b^3 (\nabla b_3), \quad b^1 (\nabla b^4 ) = \sqrt{2} b^2(\nabla b_7), \quad b^i (\nabla b_1) =0, \quad i=2,3 \hspace{.5em}.
\end{align*}
The remaining relations are either fulfilled trivially or identically zero, i.e. they do not contain any further information. The relations mentioned above lead to a system of partial differential equations given by \\[\baselineskip]
\noindent $b^1 (\nabla b_1) = b^2 (\nabla b_2) +b^3 (\nabla b_3) : $ \hfill
\begin{align}
(u)_{x_6} + (r_6)_{x_1} &= 2 \cdot (s_6)_{x_2} + (s_7)_{x_3} + (t_6)_{x_3} \label{Eq:GL1a1}\\
 (u)_{x_7} + (r_7)_{x_1} &= 2 \cdot (t_7)_{x_3} + (s_7)_{x_2} + (t_6)_{x_2} \label{Eq:GL1a2}
\end{align}
$b^1 (\nabla b_4)  = \sqrt{2}\cdot b^2 (\nabla b_7) :$  \hfill 
\begin{align}
 \sqrt{2} (r_5)_{x_4} &= u\cdot \left[ -(r_7)_{x_6} + (r_6)_{x_7} +r_6 \cdot ( r_7)_{x_1}-(r_6)_{x_1} \cdot r_7\right]-(s_7)_{x_5} + (t_6)_{x_5} \nonumber\\
 (s_7)_{x_2}&= (t_6)_{x_2} \nonumber\\
(s_7)_{x_3}& = (t_6)_{x_3} \nonumber \\
 0 &= -(s_7)_{x_6} + (s_6)_{x_7} + t_6 \cdot (s_7)_{x_3} + s_6 \cdot (s_7)_{x_2}- t_7 \cdot (s_6)_{x_3} - s_7 \cdot (s_6)_{x_2} \label{Eq:GL1a3}\\
 0 &=  -(t_7)_{x_6} + (t_6)_{x_7} + t_6 \cdot (t_7)_{x_3} + s_6 \cdot (t_7)_{x_2}- t_7 \cdot (t_6)_{x_3} - s_7 \cdot (t_6)_{x_2} \label{Eq:GL1a4}
\end{align}
$ b^2 (\nabla b_1) = 0 :$\hfill
\begin{align*}
 (u)_{x_6} &= (r_6)_{x_1} \phantom{\hspace{.5em}.}
\end{align*}
$ b^3 (\nabla b_1) = 0 :$\hfill
\begin{align*}
  (u)_{x_7} &= (r_7)_{x_1} \hspace{.5em}.
 \end{align*}
 After some substitutions and integrations this system give rise to the dependencies of $r_5,r_6$, $r_7,s_7$ on local coordinates as proposed in the theorem. The functions $F_4,F_5,F_6,F_7$ are appropriate `integration constants' and the Eqs. \eqref{Eq:GL1a1} - \eqref{Eq:GL1a4} are equivalent to the system of conditional equations for $u,s_6,t_6,t_7,F_4$ as stated.
\end{proof}
\begin{example}
One can easily check by direct calculation that the functions 
\begin{align*}
u(x_5,x_6,x_7) &= \frac{4}{3}x_6 + \frac{4x_7}{1+\mathrm{e}^{x_5}} \\
r_5(x_4,x_5,x_6,x_7) &= \frac{1}{\sqrt{2}}\frac{x_4 x_6 \cdot (4x_7 + 1 + \mathrm{e}^{x_5})}{1+\mathrm{e}^{x_5}} \cdot \exp\left( \frac{4}{3}\frac{3x_7 + x_6\left(1+ \mathrm{e}^{x_5}\right)}{1+\mathrm{e}^{x_5}}\right)\\
r_6(x_1,x_5,x_6,x_7) &= \frac{4}{3}x_1 + x_6 x_7\\
r_7(x_1,x_5,x_6,x_7) &= \frac{4x_1}{1+\mathrm{e}^{x_5}} \\
s_6(x_2,x_3,x_5,x_6,x_7) &= \frac{1}{3} \left[x_2 + x_3 \cdot \left( 1+\mathrm{e}^{x_5}\right)\right]\\
s_7(x_2,x_3,x_5,x_6,x_7) &= \frac{x_2 + x_3 \cdot \left( 1+\mathrm{e}^{x_5}\right)}{1+\mathrm{e}^{x_5}} \\
t_6(x_2,x_3,x_5,x_6,x_7) &=\frac{x_2 + x_3 \cdot \left( 1+\mathrm{e}^{x_5}\right)}{1+\mathrm{e}^{x_5}} \\
t_7(x_2,x_3,x_5,x_6,x_7) &= \frac{3 \left[ x_2 + x_3 \cdot \left( 1+\mathrm{e}^{x_5}\right)\right]}{\left(1+\mathrm{e}^{x_5}\right)^2}  \\
F_6(x_5,x_6,x_7) &= x_6 x_7 \qquad F_4 = F_5 = F_7 = 0
\end{align*}
satisfy the Eqs. \eqref{Eq:GL1a1} - \eqref{Eq:GL1a4}, and the metric defined by them has holonomy $\mathfrak{h} = \mathfrak{gl}(2,\mathbb{R}) \ltimes \mathfrak{m}(1,2) \subset \mathfrak{h}^{I\!I\!I}$.
\end{example}
%
%
%
\subsection{Holonomy of Type III 1b}
\begin{prop}\label{Th:Case1b}
The holonomy of $(M^{4,3},g)$ is contained in $\mathrm{G}_2^{*}$ and is of Type III 1a if and only if there are local coordinates $x_1,\ldots,x_7$ such that $g = 2(b^1\cdot b^5 + b^2 \cdot b^6 + b^3\cdot b^7) - (b^4)^2$, where
\begin{align*}
b^{1} &= \mathrm{d}{x_1} + r_5(x_4,x_5,x_6,x_7) \cdot  \mathrm{d}{x_5} + r_6(x_5,x_6,x_7) \cdot  \mathrm{d}{x_6} + r_7(x_5,x_6,x_7) \cdot  \mathrm{d}{x_7}\\
b^{2} &= \mathrm{d}{x_2} + s_6(x_2,x_3,x_5,x_6,x_7)  \cdot  \mathrm{d}{x_6}  + s_7(x_2,x_3,x_5,x_6,x_7)  \cdot  \mathrm{d}{x_7}\\
b^{3} &= \mathrm{d}{x_3} +t_6(x_2,x_3,x_5,x_6,x_7)  \cdot  \mathrm{d}{x_6}+ t_7(x_2,x_3,x_5,x_6,x_7)   \cdot  \mathrm{d}{x_7}\\
b^{i} &= \mathrm{d}{x_i} \qquad \qquad i = 4,5,6,7,
\end{align*}
with functions $r_5,s_7$ given by
\begin{align*}
r_5 &= \frac{x_4}{\sqrt{2}}\left[( r_6(x_5,x_6,x_7))_{x_7} - ( r_7(x_5,x_6,x_7))_{x_6} - (F_4(x_5,x_6,x_7))_{x_5}\right]+ \!F_5(x_5,x_6,x_7)\\
s_7 &= t_6 + F_4(x_5,x_6,x_7)
\end{align*}
for arbitrary functions $F_5,r_6,r_7$ and functions $s_6,t_6,t_7,F_4$ satisfying the differential equations
\begin{align*}
 (s_6)_{x_2}&=  -(t_6)_{x_3}\\
 (t_7)_{x_3}  &= -(t_6)_{x_2}\\
  (s_7)_{x_6}+t_7 \cdot (s_6)_{x_3} + s_7 \cdot (s_6)_{x_2}&=  (s_6)_{x_7} + t_6 \cdot (s_7)_{x_3} + s_6 \cdot (s_7)_{x_2}   \\
 (t_7)_{x_6}+ t_7 \cdot (t_6)_{x_3} + s_7 \cdot (t_6)_{x_2}  &=   (t_6)_{x_7} + t_6 \cdot (t_7)_{x_3} + s_6 \cdot (t_7)_{x_2}\hspace{.5em}.
\end{align*}
\end{prop}
\begin{proof}
The proof is essentially the same as in case 1a. The structure equation
\begin{align*}
\begin{pmatrix}
\mathrm{d}{b^{1}} \\ \mathrm{d}{b^{2}} \\ \mathrm{d}{b^{3}} \\ \mathrm{d}{b^{4}} \\ \mathrm{d}{b^{5}} \\ \mathrm{d}{b^{6}} \\ \mathrm{d}{b^{7}}
\end{pmatrix} =
-\begin{pmatrix}
0 & 0 & 0 & \sqrt{2} \mathbf{v} & 0 & -\mathbf{y}_{1} & -\mathbf{y}_{2} \\
0 &  \mathbf{a}_1 &  \mathbf{a}_2 & 0  & \mathbf{y}_{1}& 0 & \mathbf{v} \\
0 &  \mathbf{a}_3 &  -\mathbf{a}_1 & 0  & \mathbf{y}_{2}& - \mathbf{v} & 0 \\
0 & 0 & 0 & 0 & \sqrt{2} \mathbf{v}& 0 &0 \\
0 & 0 & 0 & 0 & 0 & 0 &0 \\
0 & 0 & 0 & 0 & 0 &- \mathbf{a}_1  &  -\mathbf{a}_3 \\
0 & 0 & 0 & 0 & 0 &  -\mathbf{a}_2 & \mathbf{a}_1  \\
\end{pmatrix} \wedge 
\begin{pmatrix}
b^{1} \\ b^{2} \\ b^{3} \\ b^{4} \\ b^{5} \\ b^{6} \\ b^{7}
\end{pmatrix}
\end{align*}
implies the choice of $b^5 = \mathrm{d}x_5$, since $\mathrm{d}b^5 = 0$. Thus, set $u\equiv 0$ in the proof of Proposition \ref{Th:Case1a}. Apart from that, following the same considerations as before proves Proposition \ref{Th:Case1b}.
\end{proof}
\begin{example}
One can easily check that the functions 
\begin{align*}
r_5(x_4,x_5,x_6,x_7) &= \frac{1}{\sqrt{2}}x_4 x_6 \qquad r_6(x_5,x_6,x_7) =  x_6 x_7 \qquad r_7(x_5,x_6,x_7) = 0\\
s_6(x_2,x_3,x_5,x_6,x_7) &= -\left[x_2 + x_3 \cdot \left( 1+\mathrm{e}^{x_5}\right)\right]\\
s_7(x_2,x_3,x_5,x_6,x_7) &= \frac{x_2 + x_3 \cdot \left( 1+\mathrm{e}^{x_5}\right)}{1+\mathrm{e}^{x_5}} \\
t_6(x_2,x_3,x_5,x_6,x_7) &=\frac{x_2 + x_3 \cdot \left( 1+\mathrm{e}^{x_5}\right)}{1+\mathrm{e}^{x_5}} \\
t_7(x_2,x_3,x_5,x_6,x_7) &= -\frac{ \left[ x_2 + x_3 \cdot \left( 1+\mathrm{e}^{x_5}\right)\right]}{\left(1+\mathrm{e}^{x_5}\right)^2}\\
\end{align*}
satisfy the conditional equations of the theorem and the metric defined that way has ho\-lo\-no\-my $\mathfrak{h} = \mathfrak{sl}(2,\mathbb{R}) \ltimes \mathfrak{m}(1,2) \subset \mathfrak{h}^{I\!I\!I}$.
\end{example}
%
%
%
\subsection{Holonomy of Type III 1c}
\begin{prop}
The holonomy of $(M^{4,3},g)$ is contained in $\mathrm{G}_2^{*}$ and is of Type III 1c if and only if there are local coordinates $x_1,\ldots,x_7$ such that $g = 2(b^1\cdot b^5 + b^2 \cdot b^6 + b^3\cdot b^7) - (b^4)^2$, where
\begin{align*}
b^{1} &= \mathrm{d}{x_1} + r_5(x_4,x_5,x_6,x_7) \cdot  \mathrm{d}{x_5} + r_6(x_1,x_5,x_6,x_7) \cdot  \mathrm{d}{x_6} + r_7(x_1,x_5,x_6,x_7) \cdot  \mathrm{d}{x_7}\\
b^{2} &= \mathrm{d}{x_2} + s_6(x_2,x_3,x_5,x_6,x_7)  \cdot  \mathrm{d}{x_6}  + s_7(x_2,x_3,x_5,x_6,x_7)  \cdot  \mathrm{d}{x_7}\\
b^{3} &= \mathrm{d}{x_3} +t_6(x_2,x_3,x_5,x_6,x_7)  \cdot  \mathrm{d}{x_6}+ t_7(x_2,x_3,x_5,x_6,x_7)   \cdot  \mathrm{d}{x_7}\\
b^{5} & = \mathrm{e}^{u(x_5,x_6,x_7)} \cdot \mathrm{d}{x_5}\\
b^{6} & =\mathrm{d}{x_6}+ v(x_6,x_7) \cdot \mathrm{d}{x_7}\\
b^{i} &= \mathrm{d}{x_i} \qquad \qquad i = 4,7,
\end{align*}
with functions $r_5, r_6, r_7,t_6,s_6,s_7, t_7 $ given by
\begin{align*}
r_5 &= \frac{x_4}{\sqrt{2}} \left[ \left(\mathrm{e}^u \cdot F_6  (x_5,x_6,x_7)\right)_{x_7} - \left(\mathrm{e}^u \cdot F_7  (x_5,x_6,x_7)\right)_{x_6} - (F_4 (x_5,x_6,x_7))_{x_5} \right]\\
&\hspace*{9cm} + F_5 (x_5,x_6,x_7) \\
r_6 &= x_1 \cdot  (u)_{x_6} +   F_6  (x_5,x_6,x_7) \\
r_7 &= x_1\cdot  (u)_{x_7} +   F_7  (x_5,x_6,x_7) \\
s_6 &=  -\frac{x_3}{2}\cdot \left( (u)_{x_7} - v\cdot  (u)_{x_6}+ 2(v)_{x_6}\right) + \frac{x_2}{2}\cdot (u)_{x_6} +F_1(x_5,x_6,x_7)\\
t_6 &= \frac{x_2}{2}\cdot \left( (u)_{x_7} - v\cdot  (u)_{x_6}+2(v)_{x_6}\right) + \frac{x_3}{2}\cdot (u)_{x_6} +F_2(x_5,x_6,x_7) \\
s_7 &=  t_6 + v\cdot s_6 -x_2 \cdot (v)_{x_6} + F_3(x_5,x_6,x_7)\\
t_7 &=  \frac{x_3}{2}\cdot  (u)_{x_7} - \frac{x_2}{2}\cdot  (u)_{x_6} + v \cdot t_6 + F_4(x_5,x_6,x_7)
\end{align*}
for arbitrary functions $F_5,F_6,F_7$ and functions $u,v,F_1,\ldots,F_4$ satisfying the differential equations
\begin{align}
0 &= -(s_7)_{x_6} + (s_6)_{x_7} + t_6 \cdot (s_7)_{x_3} + s_6 \cdot (s_7)_{x_2}- t_7 \cdot (s_6)_{x_3} - s_7 \cdot (s_6)_{x_2} \label{Eq:GL1cCond1}\\
 0 &=  -(t_7)_{x_6} + (t_6)_{x_7} + t_6 \cdot (t_7)_{x_3} + s_6 \cdot (t_7)_{x_2}- t_7 \cdot (t_6)_{x_3} - s_7 \cdot (t_6)_{x_2} \label{Eq:GL1cCond2}\hspace{.5em}.
\end{align}
\end{prop}
\begin{proof}
The structure equation
\begin{align*}
\begin{pmatrix}
\mathrm{d}{b^{1}} \\ \mathrm{d}{b^{2}} \\ \mathrm{d}{b^{3}} \\ \mathrm{d}{b^{4}} \\ \mathrm{d}{b^{5}} \\ \mathrm{d}{b^{6}} \\ \mathrm{d}{b^{7}}
\end{pmatrix} =
-\begin{pmatrix}
2\textbf{a} & 0 & 0 & \sqrt{2} \textbf{v} & 0 & -\textbf{y}_{1} & -\textbf{y}_{2} \\
0 &  \textbf{a} &-  \textbf{b} & 0  & \textbf{y}_{1}& 0 & \textbf{v} \\
0 &  \textbf{b} &  \textbf{a} & 0  & \textbf{y}_{2}& - \textbf{v} & 0 \\
0 & 0 & 0 & 0 & \sqrt{2} \textbf{v}& 0 &0 \\
0 & 0 & 0 & 0 &- 2\textbf{a} & 0 &0 \\
0 & 0 & 0 & 0 & 0 &- \textbf{a}  &  -\textbf{b} \\
0 & 0 & 0 & 0 & 0 &  \textbf{b} & -\textbf{a}  \\
\end{pmatrix} \wedge 
\begin{pmatrix}
b^{1} \\ b^{2} \\ b^{3} \\ b^{4} \\ b^{5} \\ b^{6} \\ b^{7}
\end{pmatrix}
\end{align*}
implies the existence of local coordinates $y_6,x_7$ such that $b^6,b^7$ lie in the ideal generated by $dy_6,\mathrm{d}x_7$. Thus, assume
\begin{align*}
b^6 = k_1 \cdot \mathrm{d}y_6 + k_2 \cdot \mathrm{d}x_7 \qquad b^7 = k_3 \cdot \mathrm{d}y_6 + k_4 \cdot \mathrm{d}x_7
\end{align*}
for suitable functions $k_1,\ldots,k_4$. Transforming the basis pointwise by an appropriate element 
\begin{align*}
 \mathrm{GL}(2, \mathbb{R}) \ni M = \begin{pmatrix} m_1 & -m_2 \\ m_2 & m_1 \end{pmatrix}   \cong  \mathrm{diag}(\det M, M,1, (\det M)^{-1},(M^{\top})^{-1}) \in \mathrm{G}_2^*
\end{align*}
yields $b^7 = \mathrm{d}x_7$ and $b^6 = g_1 \cdot \mathrm{d}y_6 + g_2 \cdot \mathrm{d}x_7$ for some new functions $g_1,g_2$. Since $\mathrm{d}b^7=\mathrm{dd}x_7 = 0$, $\mathbf{a}$ and $\mathbf{b}$ lie in $\mathcal{I}(\mathrm{d}y_6,\mathrm{d}x_7)$, the ideal generated by $\mathrm{d}y_6$ and $\mathrm{d}x_7$, and $g_1,g_2$ depend on $y_6,x_7$ only. Therefore, it is possible to introduce a new local coordinate $x_6$ by 
\begin{align*}
x_6 = \int_{0}^{y_6} g_1(\tau,x_7) \;d\tau
\end{align*}
such that $b^6 = \mathrm{d}x_6 + v(x_6,x_7) \cdot \mathrm{d}x_7$ with an appropriate function $v$. Consequently, one has $\mathbf{a},\mathbf{b}$ $\in \mathcal{I}(\mathrm{d}x_6,\mathrm{d}x_7)$.

The structure equation for $\mathrm{d}b^5$ allows the choice $b^5 = \mathrm{e}^{u} \cdot \mathrm{d}x_5$ for a suitable local function $u$ and coordinate $x_5$. Since $\mathbf{a} \in \mathcal{I}(\mathrm{d}x_6,\mathrm{d}x_7)$, it follows $u = u(x_5,x_6,x_7)$. 

From the structure equation deduce $b^4 = \mathrm{d}x_4 + q \cdot \mathrm{d}x_5$. After a change of basis using a suitable element 
$\exp(h(0,v,0))$ obtain $b^4 =\mathrm{d}x_4$ and $\mathrm{d}b^4 = 0$. Hence, $\mathbf{v} \in \mathcal{I}(\mathrm{d}x_5)$.

Next, the structure equation implies $\mathrm{d}b^2,\mathrm{d}b^3 \in \mathcal{I}(\mathrm{d}x_5,\mathrm{d}x_6,dx_7)$, as $\mathbf{a}$, $\mathbf{b}$ $\in \mathcal{I}(\mathrm{d}x_6,\mathrm{d}x_7)$, $\mathbf{v} \in \mathcal{I}(\mathrm{d}x_5)$ and $b^5 = \mathrm{e}^{u} \cdot \mathrm{d}x_5$. It follows
\begin{align*}
b^2 &= \mathrm{d}{x_2} + s_5 \cdot  \mathrm{d}{x_5}+ s_6 \cdot  \mathrm{d}{x_6}+ s_7 \cdot  \mathrm{d}{x_7}\\
b^3 &= \mathrm{d}{x_3} + t_5 \cdot  \mathrm{d}{x_5}+ t_6 \cdot  \mathrm{d}{x_6}+ t_7 \cdot  \mathrm{d}{x_7} \hspace{.5em},
\end{align*}
which reduces to 
\begin{align*}
b^2 &= \mathrm{d}{x_2} +  s_6 \cdot  \mathrm{d}{x_6}+ s_7 \cdot  \mathrm{d}{x_7} \\
b^3 &= \mathrm{d}{x_3} +  t_6 \cdot  \mathrm{d}{x_6}+ t_7 \cdot  \mathrm{d}{x_7}
\end{align*}
after transforming the basis with a suitable element $\exp(h(0,0,(y_1,y_2)^{\top}))$. Comparing both sides of the equations 
\begin{align*}
 \mathrm{d}{b^{2}} &= \mathrm{d}{s_6} \wedge  \mathrm{d}{x_6}+\mathrm{d}{s_7} \wedge  \mathrm{d}{x_7} = - \textbf{a} \wedge b^{2} + \textbf{b} \wedge b^{3} -\textbf{y}_1 \wedge b^{5} - \textbf{v} \wedge b^{7} \\
  \mathrm{d}{b^{3}} &= \mathrm{d}{t_6} \wedge  \mathrm{d}{x_6}+\mathrm{d}{t_7} \wedge  \mathrm{d}{x_7} = - \textbf{b} \wedge b^{2}- \textbf{a} \wedge b^{3} -\textbf{y}_2 \wedge b^{5}+\textbf{v} \wedge b^{6} 
\end{align*}
yields $\mathbf{y}_1,\mathbf{y}_2\in \mathcal{I}(\mathrm{d}x_5,\mathrm{d}x_6,\mathrm{d}x_7)$ and 
\begin{align*}
s_6 &= s_6 (x_2,x_3,x_5,x_6,x_7), \qquad s_7 = s_7 (x_2,x_3,x_5,x_6,x_7),\\
  t_6 &= t_6 (x_2,x_3,x_5,x_6,x_7), \qquad \;t_7 = t_7 (x_2,x_3,x_5,x_6,x_7) \hspace{.5em}.
\end{align*}
An analogue consideration for $b^1$, keeping $\mathbf{y}_1,\mathbf{y}_2\in \mathcal{I}(\mathrm{d}x_5,\mathrm{d}x_6,\mathrm{d}x_7)$ in mind, leads to 
\begin{align*}
b^1 = \mathrm{d}{x_1} + r_5(x_4,x_5,x_6,x_7) \cdot  \mathrm{d}{x_5}+ r_6(x_4,x_5,x_6,x_7) \cdot  \mathrm{d}{x_6}+r_7(x_4,x_5,x_6,x_7) \cdot  \mathrm{d}{x_7} \hspace{.2em}.
\end{align*}
Again, the metric $g$ is used, with respect to the local coordinates $x_1,\ldots,x_7$ considered so far, in order to calculate the components of the local connection form $\theta$. The holonomy of $g$ is contained in $\mathfrak{co}(2) \ltimes \mathfrak{m}(1,2)$ if and only if the following relations between the components hold
\begin{align*}
 b^1 (\nabla b_1) &= b^2 (\nabla b_2) +b^3 (\nabla b_3), \qquad
  b^2 (\nabla b_2)  = b^3 (\nabla b_3), \qquad
 b^3 (\nabla b_2)  = -b^2 (\nabla b_3), \\
 b^1 (\nabla b_4)  &= \sqrt{2}\cdot b^2 (\nabla b_7), \qquad
 b^i (\nabla b_1) = 0 \quad i = 2,3 \hspace{.5em}.
\end{align*}
The remaining relations are either trivial or identically satisfied. The relations give rise to the following system of partial differential equations\\[\baselineskip]
\noindent $b^1 (\nabla b_1) = b^2 (\nabla b_2) +b^3 (\nabla b_3) : $ \hfill
 \begin{align}
 (u)_{x_6} + (r_6)_{x_1} &= 2 \cdot (s_6)_{x_2} + (s_7)_{x_3} + (t_6)_{x_3}-v\cdot (s_6)_{x_3}\nonumber \\
 \begin{split}
 (u)_{x_7} -v\cdot  (u)_{x_7}+ (r_7)_{x_1} - v \cdot  (r_6)_{x_1} &= 2 \cdot \left( (t_7)_{x_3} -v \cdot (t_6)_{x_3} \right)\\ &\hspace{1cm}+ (s_7)_{x_2} + (t_6)_{x_2} -v \cdot (s_6)_{x_2} - (v)_{x_6} 
 \end{split}\nonumber
 \end{align}
$ b^2 (\nabla b_2)  = b^3 (\nabla b_3)  :$ \hfill
\begin{align}
 2 \cdot (s_6)_{x_2} &= (s_7)_{x_3} + (t_6)_{x_3}-v\cdot (s_6)_{x_3}\nonumber\\
  2 \cdot \left( (t_7)_{x_3} -v \cdot (t_6)_{x_3} \right)&=  (s_7)_{x_2} + (t_6)_{x_2} -v \cdot (s_6)_{x_2} - (v)_{x_6} v \nonumber
\end{align}
$b^3 (\nabla b_2)  = -b^2 (\nabla b_3) :$ \hfill
\begin{align}
  -2 \cdot (s_6)_{x_3} &= (s_7)_{x_2} + (t_6)_{x_2} -v\cdot (s_6)_{x_2} + (v)_{x_6} \nonumber\\
 -2 \cdot \left( (t_7)_{x_2} -v \cdot (t_6)_{x_2} \right)&= (s_7)_{x_3} + (t_6)_{x_3} -v\cdot (s_6)_{x_3}\nonumber
\end{align}
$ b^1 (\nabla b_4)  = \sqrt{2} b^2 (\nabla b_7) :$ \hfill
\begin{align}
\sqrt{2} (r_5)_{x_4}& = u\cdot \left[ -(r_7)_{x_6} + (r_6)_{x_7} +r_6 \cdot ( r_7)_{x_1}-(r_6)_{x_1} \cdot r_7\right]\\
&\hspace{12em}-(s_7)_{x_5} + (t_6)_{x_5}+v\cdot (s_6)_{x_5}\nonumber\\
 (s_7)_{x_2} &= (t_6)_{x_2}+v\cdot (s_6)_{x_2}-(v)_{x_6}  \nonumber\\
 (s_7)_{x_3} &= (t_6)_{x_3}+v\cdot (s_6)_{x_3} \nonumber\\
 0 &= -(s_7)_{x_6} + (s_6)_{x_7} + t_6 \cdot (s_7)_{x_3} + s_6 \cdot (s_7)_{x_2}- t_7 \cdot (s_6)_{x_3} - s_7 \cdot (s_6)_{x_2}  \label{Eq:GL1c1}\\
 0 &=  -(t_7)_{x_6} + (t_6)_{x_7} + t_6 \cdot (t_7)_{x_3} + s_6 \cdot (t_7)_{x_2}- t_7 \cdot (t_6)_{x_3} - s_7 \cdot (t_6)_{x_2}  \label{Eq:GL1c2}%
\end{align}
$ b^2 (\nabla b_1) = 0 :$
\begin{align*} 
  (u)_{x_6} = (r_6)_{x_1} \nonumber
\end{align*}
$ b^3 (\nabla b_1) = 0 :$
\begin{align*}
  (u)_{x_7} -v\cdot  (u)_{x_6} + v \cdot (r_6)_{x_1}&= (r_7)_{x_1} \hspace{.5em}.
\end{align*} 
 The system can be solved by substitution and integration. Along with the integration one has to introduce new function $F_1,\ldots,F_7$. Finally the system is equivalent to the form of the functions $r_5,r_6,r_7,t_6,t_7,s_6,s_7$ as stated in the proposition, except the Eqs. \eqref{Eq:GL1c1} and \eqref{Eq:GL1c2} which are equivalent to the conditional equations \eqref{Eq:GL1cCond1} and \eqref{Eq:GL1cCond2}. 
\end{proof}
\begin{example}
It is easy to verify that the functions given by the choice
\begin{align*}
u(x_5,x_6,x_7) = x_5 x_6 \qquad F_6(x_5,x_6,x_7) = x_5 x_6 x_7 \qquad v = F_1 = \ldots = F_5 = F_7 = 0
\end{align*}
satisfy the conditional equations, and that the holonomy, with respect to the metric $g$ as defined, is $\mathfrak{h} = \mathfrak{co}(2) \ltimes \mathfrak{m}(1,2)$.
\end{example}
\subsection{Holonomy of Type III 1d}
\begin{prop}
The holonomy of $(M^{4,3},g)$ is contained in $\mathrm{G}_2^{*}$ and is of Type III 1d if and only if there are local coordinates $x_1,\ldots,x_7$ such that $g = 2(b^1\cdot b^5 + b^2 \cdot b^6 + b^3\cdot b^7) - (b^4)^2$, where
\begin{align*}
b^{1} &= \mathrm{d}{x_1} + r_5(x_4,x_5,x_6,x_7) \cdot  \mathrm{d}{x_5} + r_6(x_1,x_5,x_6,x_7) \cdot  \mathrm{d}{x_6} + r_7(x_1,x_5,x_6,x_7) \cdot  \mathrm{d}{x_7}\\
b^{2} &= \mathrm{d}{x_2} + s_6(x_2,x_5,x_6)  \cdot  \mathrm{d}{x_6} \\
b^{3} &= \mathrm{d}{x_3} + t_7(x_3,x_5,x_7)   \cdot  \mathrm{d}{x_7}\\
b^{5} & = \mathrm{e}^{u(x_5,x_6,x_7)} \cdot \mathrm{d}{x_5}\\
b^{i} &= \mathrm{d}{x_i} \qquad \qquad i = 4,6,7 \hspace{.5em},
\end{align*}
with functions $r_5, r_6, r_7,s_6, t_7 $ given by
\begin{align*}
r_5 &= \frac{x_4}{\sqrt{2}} \left[ \left(\mathrm{e}^{u} \cdot F_6  (x_5,x_6,x_7)\right)_{x_7} - \left(\mathrm{e}^{u} \cdot F_7  (x_5,x_6,x_7)\right)_{x_6} \right] + F_5 (x_5,x_6,x_7) \\
r_6 &= x_1 \cdot (u)_{x_6} +   F_6  (x_5,x_6,x_7) \\
r_7 &= x_1 \cdot (u)_{x_7} +   F_7  (x_5,x_6,x_7) \\
s_6 &= x_2\cdot(u)_{x_6} + F_1 (x_5,x_6)\\
t_7 &= x_3\cdot(u)_{x_7} + F_2 (x_5,x_7)
\end{align*}
for arbitrary functions $F_1,F_2,F_4,\ldots,F_7$ and a function $u$ satisfying the equations
\begin{align*}
(u)_{x_6} = K_1(x_5,x_6) \qquad (u)_{x_7} = K_2(x_5,x_7) \hspace{.5em},
\end{align*}
where $K_1,K_2$ are arbitrary functions.
\end{prop}
\begin{proof}
The structure equation
\begin{align*}
\begin{pmatrix}
\mathrm{d}{b^{1}} \\ \mathrm{d}{b^{2}} \\ \mathrm{d}{b^{3}} \\ \mathrm{d}{b^{4}} \\ \mathrm{d}{b^{5}} \\ \mathrm{d}{b^{6}} \\ \mathrm{d}{b^{7}}
\end{pmatrix} =
-\begin{pmatrix}
 \textbf{a}+\textbf{d} & 0 & 0 & \sqrt{2} \textbf{v} & 0 & -\textbf{y}_{1} & -\textbf{y}_{2} \\
0 &  \textbf{a} & 0 & 0  & \textbf{y}_{1}& 0 & \textbf{v} \\
0 & 0 &  \textbf{d} & 0  & \textbf{y}_{2}& - \textbf{v} & 0 \\
0 & 0 & 0 & 0 & \sqrt{2} \textbf{v}& 0 &0 \\
0 & 0 & 0 & 0 &- \textbf{a}-\textbf{d} & 0 &0 \\
0 & 0 & 0 & 0 & 0 &- \textbf{a}  & 0 \\
0 & 0 & 0 & 0 & 0 & 0 & -\textbf{d}  \\
\end{pmatrix} \wedge 
\begin{pmatrix}
b^{1} \\ b^{2} \\ b^{3} \\ b^{4} \\ b^{5} \\ b^{6} \\ b^{7}
\end{pmatrix}
\end{align*}
allows the choice of $b^6 = \mathrm{d}x_6$ and $b^7 = \mathrm{d}x_7$ after changing the basis with a suitable element
\begin{align*}
 \mathrm{GL}(2,\mathbb{R}) \ni M = \begin{pmatrix} m_1 & 0 \\ 0 & m_2 \end{pmatrix} \cong  \mathrm{diag}(\det M, M,1, (\det M)^{-1},(M^{\top})^{-1}) \in \mathrm{G}_2^* \hspace{.5em}.
\end{align*}
Therefore, $\mathbf{a} \in \mathcal{I}(\mathrm{d}x_6)$ and $\mathbf{d} \in \mathcal{I}(\mathrm{d}x_7)$. The equation for $\mathrm{d}b^5$ implies $b^5 = \mathrm{e}^{u} \cdot \mathrm{d}x_5$ for a suitable local function $u$ and coordinate $x_5$. Then, $(\mathbf{a}+\mathbf{d}) \in \mathcal{I}(\mathrm{d}x_6,\mathrm{d}x_7)$ implies $u = u(x_5,x_6,x_7)$.

Next, the equation for $\mathrm{d}b^4$ implies $b^4 = \mathrm{d}x_4 + q \cdot \mathrm{d}x_5$, and with a transformation using a suitable element $\exp(h(0,v,0))$ one arrives at $b^4 = \mathrm{d}x_4$. From $\mathrm{d}b^4 = 0$ deduce $\mathbf{v} \in \mathcal{I}(\mathrm{d}x_5)$. 

Since $\mathbf{a} \in \mathcal{I}(\mathrm{d}x_6)$, $\mathbf{d} \in \mathcal{I}(\mathrm{d}x_7)$, $\mathbf{v} \in \mathcal{I}(\mathrm{d}x_5)$ and $b^5 = \mathrm{e}^{u} \cdot \mathrm{d}x_5$, from the structure equation follows $\mathrm{d}b^2 \in \mathcal{I}(\mathrm{d}x_5,\mathrm{d}x_6)$ and $\mathrm{d}b^3 \in \mathcal{I}(\mathrm{d}x_5,\mathrm{d}x_7)$. Thus, choose 
\begin{align*}
b^2 = \mathrm{d}x_2 + s_5 \cdot \mathrm{d}x_5 + s_6 \cdot \mathrm{d}x_6 \qquad b^3 = \mathrm{d}x_3 + t_5 \cdot \mathrm{d}x_5 + t_7 \cdot \mathrm{d}x_7\hspace{.5em},
\end{align*}
and after a change of basis with an appropriate element $\exp(h(0,0,(y_1,y_2)^{\top}))$ obtain
\begin{align*}
b^2 = \mathrm{d}x_2  + s_6 \cdot \mathrm{d}x_6 \qquad b^3 = \mathrm{d}x_3  + t_7 \cdot \mathrm{d}x_7\hspace{.5em}.
\end{align*}
Comparing both sides of the equations 
\begin{align*}
\mathrm{d}{b^{2}} &= d{s_6} \wedge \mathrm{d}{x_6}  = - \textbf{a} \wedge b^{2} -\textbf{y}_1 \wedge b^{5} - \textbf{v} \wedge b^{7} \\
\mathrm{d}{b^{3}} &= \mathrm{d}{t_7} \wedge \mathrm{d}{x_7} = - \textbf{d} \wedge b^{3} -\textbf{y}_2 \wedge b^{5}+\textbf{v} \wedge b^{6} 
\end{align*}
yields $\mathbf{y}_1,\mathbf{y}_1 \in \mathcal{I}(\mathrm{d}x_5,\mathrm{d}x_6,\mathrm{d}x_7)$, $s_6=s_6(x_2,x_5,x_6)$ and $t_7 = t_7(x_3,x_5,x_7)$.

Finally the structure equation implies $\mathrm{d}b^1 \in \mathcal{I}(\mathrm{d}x_5,\mathrm{d}x_6,\mathrm{d}x_7)$ so that the assumption
\begin{align*}
b^1 = \mathrm{d}{x_1} + r_5 \cdot  \mathrm{d}{x_5}+ r_6 \cdot  \mathrm{d}{x_6}+r_7 \cdot  \mathrm{d}{x_7} \hspace{.5em}.
\end{align*}
is possible.
Comparing the expressions for $\mathrm{d}b^1$
\begin{align*}
 \mathrm{d}r_5 \wedge \mathrm{d}x_5 + \mathrm{d}r_6 \wedge \mathrm{d}x_6 + \mathrm{d}r_7 \wedge \mathrm{d}x_7 = -(\mathbf{a}+\mathbf{d}) \wedge b^1 - \sqrt{2}\mathbf{v} \wedge b^4 + \mathbf{y}_1 \wedge b^6 + \mathbf{y}_2 \wedge b^7
\end{align*}
leads to the dependencies 
\begin{align*}
r_5 = r_5( x_4,x_5,x_6,x_7) \qquad  r_6 = r_6( x_1,x_5,x_6,x_7) \qquad  r_7 = r_7( x_1,x_5,x_6,x_7) \hspace{.5em}.
\end{align*}

Using the metric $g$ with respect to the local coordinates allows the calculation of the components of the local connection form. The holonomy of $g$ is contained in $\mathfrak{d} \ltimes \mathfrak{m}(1,2)$ if and only if the following relations are satisfied:
\begin{align*}
 b^1 (\nabla b_1) &= b^2 (\nabla b_2) +b^3 (\nabla b_3), \quad
 b^1 (\nabla b_4)  = \sqrt{2} \cdot b^2 (\nabla b_7), \quad
 b^i (\nabla b_1) = 0 \quad i = 2,3 \hspace{.5em}.
\end{align*}
These relations are equivalent to the following partial differential system:\\[\baselineskip]
\noindent $b^1 (\nabla b_1) = b^2 (\nabla b_2) +b^3 (\nabla b_3) : $ \hfill
\begin{align}
  (u)_{x_6} + (r_6)_{x_1} &= 2 \cdot (s_6)_{x_2} \label{Eq:GL1d1} \\
  (u)_{x_7} + (r_7)_{x_1} &= 2 \cdot (t_7)_{x_3} \label{Eq:GL1d2}
 \end{align}
 $b^1 (\nabla b_4) = \sqrt{2} \cdot b^2 (\nabla b_7) : $ \hfill
 \begin{align}
 \sqrt{2} (r_5)_{x_4}& = \mathrm{e}^{u}\cdot \left[ -(r_7)_{x_6} + (r_6)_{x_7} +r_6 \cdot ( r_7)_{x_1}-(r_6)_{x_1} \cdot r_7\right]\nonumber
 \end{align}
  $ b^2 (\nabla b_1) = 0 :$ \hfill
 \begin{align} 
  (u)_{x_6} = (r_6)_{x_1} \nonumber
\end{align}
$ b^3 (\nabla b_1) = 0 :$ \hfill
\begin{align}
  (u)_{x_7} &= (r_7)_{x_1}\nonumber \hspace{.5em},
 \end{align}
which can be easily integrated and yield the form of the functions $r_5,r_6,r_7,s_6,t_7$ as stated in the proposition. Since $s_6$ and $t_7$ depend on $x_5,x_6$ and $x_5,x_7$ respectively, equations \eqref{Eq:GL1d1} and \eqref{Eq:GL1d2} give rise to the constraints on the derivatives of $u$.
\end{proof}
\begin{example}
With the choice 
\begin{align*}
u(x_5,x_6,x_7) = \mathrm{e}^{x_5 \cdot (x_6+x_7)} \qquad F_6(x_5,x_6,x_7) = x_5 x_6 x_7 \qquad F_1 = F_2 =F_5 =F_7 =0
\end{align*}
the conditional equations are fulfilled and the holonomy with respect to the metric defined by that functions is $\mathfrak{h} = \mathfrak{d} \ltimes \mathfrak{m}(1,2)$.
\end{example}
\subsection{Holonomy of Type III 2a}
\begin{prop}
The holonomy of $(M^{4,3},g)$ is contained in $\mathrm{G}_2^{*}$ and is of Type III 2a if and only if there are local coordinates $x_1,\ldots,x_7$ such that $g = 2(b^1\cdot b^5 + b^2 \cdot b^6 + b^3\cdot b^7) - (b^4)^2$, where
\begin{align*}
b^{1} &= \mathrm{d}{x_1} +  r_5 (x_1,x_2,x_4,x_5,x_6,x_7)\cdot \mathrm{d}{x_5} + r_6(x_2,x_5,x_6,x_7) \cdot \mathrm{d}{x_6}\\
b^{3} &= \mathrm{d}{x_3} + t_5 (x_5,x_6)\cdot \mathrm{d}{x_5} \\
b^{6} &= \mathrm{d}{x_6} + v(x_5,x_6)\cdot \mathrm{d}{x_5} \\
b^{i} &= \mathrm{d}{x_i} \qquad \qquad i = 2,4,5,7,
\end{align*}
with functions $r_5,r_6$ given by
\begin{align*}
r_5 &=-x_1\cdot (v)_{x_6}-x_2 \cdot v \cdot (v)_{x_6}-\sqrt{2} \cdot x_4 \cdot  (t_5)_{x_6} - x_7 \cdot v \cdot (t_5)_{x_6} + F_5(x_5,x_6) \\
r_6 &=-x_2 \cdot (v)_{x_6}  - x_7 \cdot (t_5)_{x_6} + F_6 (x_5,x_7)
\end{align*}
and $F_5, F_6$ are arbitrary.
\end{prop}
\begin{proof}
Obviously, the structure equation
\begin{align*}
\begin{pmatrix}
\mathrm{d}{b^{1}} \\ \mathrm{d}{b^{2}} \\ \mathrm{d}{b^{3}} \\ \mathrm{d}{b^{4}} \\ \mathrm{d}{b^{5}} \\ \mathrm{d}{b^{6}} \\ \mathrm{d}{b^{7}}
\end{pmatrix} =
-\begin{pmatrix}
\textbf{x} & 0 & 0 & \sqrt{2} \textbf{v} & 0 & -\textbf{y}_{1} & 0 \\
0 & \textbf{x} & 0 & 0  & \textbf{y}_{1}& 0 & \textbf{v} \\
0 & 0 & 0 & 0  & 0& - \textbf{v} & 0 \\
0 & 0 & 0 & 0 & \sqrt{2} \textbf{v}& 0 &0 \\
0 & 0 & 0 & 0 & -\textbf{x} & 0 &0 \\
0 & 0 & 0 & 0 & 0 & -\textbf{x} & 0 \\
0 & 0 & 0 & 0 & 0 & 0 & 0 \\
\end{pmatrix} \wedge 
\begin{pmatrix}
b^{1} \\ b^{2} \\ b^{3} \\ b^{4} \\ b^{5} \\ b^{6} \\ b^{7}
\end{pmatrix}
\end{align*}
 allows the choice of $b^7 = \mathrm{d}x_7$. Furthermore, it implies $b^5 = u \cdot \mathrm{d}x_5$ and $b^6 = v' \cdot \mathrm{d}x_6$ for appropriate functions $u,v'$ and local coordinates $x_5,x_6$. Changing the basis by an element 
\begin{align*}
 \mathrm{GL}(2,\mathbb{R}) \ni M = \begin{pmatrix} u^{-1} & 0 \\ 0 & 1 \end{pmatrix}  \cong \mathrm{diag}\left(u^{-1}, M, 1, u, (M^{\top})^{-1}\right) \in \mathrm{G}_2^{*} 
\end{align*}
one obtains $b^5 = \mathrm{d}x_5$ and $b^6 =\mathrm{d}x_5+ v \cdot \mathrm{d}x_5$ with a new function $v$. Since $\mathrm{d}b^5 = 0$, it follows that $\mathbf{x} \in \mathcal{I}(\mathrm{d}x_5)$. Thus, $v = v(x_5,x_6)$. 

Next, the structure equation allows the choice $b^4 = \mathrm{d}x_4 + q \cdot \mathrm{d}x_5$. Transforming with a suitable element $\exp(h(0,v,0))$ results in $b^4 = \mathrm{d}x_4$ and therefore $\mathbf{v} \in \mathcal{I}(\mathrm{d}x_5)$.

Similar, the equation for $\mathrm{d}b^3$ implies $\mathrm{d}b^3 \in \mathcal{I}(\mathrm{d}x_5)$, since $\mathbf{v} \in \mathcal{I}(\mathrm{d}x_5)$, which is why it is possible to choose $b^3 = \mathrm{d}x_3 + t_5 \cdot \mathrm{d}x_5$. Comparing both sides of 
\begin{align*}
\mathrm{d}{b^3} = \mathrm{d}{t_5} \wedge \mathrm{d}{x_5} = \mathbf{v} \wedge ( \mathrm{d}{x_6} + v \cdot \mathrm{d}{x_5})
\end{align*}
gives $t_5 = t_5(x_5,x_6)$. From the structure equation deduce $\mathrm{d}b^2 \in \mathcal{I}(\mathrm{d}x_5)$, as $\mathbf{x},\mathbf{v} \in \mathcal{I}(\mathrm{d}x_5)$. This allows the choice 
\begin{align*}
b^2 = \mathrm{d}x_2 + s_5 \cdot \mathrm{d}x_5
\end{align*}
and transforming the basis with an appropriate element $\exp(h(0,0,(y_1,0)^{\top}))$ leads to $b^2 = \mathrm{d}x_2$. 

Finally, the structure equation implies, since $b^6 \in \mathcal{I}(\mathrm{d}x_5,\mathrm{d}x_6)$ and $\mathbf{x},\mathbf{v}\in \mathcal{I}(\mathrm{d}x_5)$, $\mathrm{d}b^1 \in \mathcal{I}(\mathrm{d}x_5,\mathrm{d}x_6)$ such that the choice
\begin{align*}
b^1 = \mathrm{d}x_1 + r_5 \cdot \mathrm{d}x_5 + r_6 \cdot \mathrm{d}x_6 \hspace{.5em}.
\end{align*}
is allowed.
The dependencies of the functions $r_5,r_6$ on the local coordinates are due to
\begin{align*}
\mathrm{d}b^1 = \mathrm{d}r_5 \wedge \mathrm{d}x_5 + \mathrm{d}r_6 \wedge \mathrm{d}x_6 = -\mathbf{x} \wedge b^1 - \sqrt{2}\mathbf{v} \wedge b^4 + \mathbf{y}_1 \wedge b^6
\end{align*}
and result in $r_5 = r_5(x_1,x_2,x_4,x_5,x_6,x_7)$ and $r_6 = r_6(x_2,x_5,x_6,x_7)$.

Now, the metric $g$, with respect to the local coordinates, is used to calculate the components of the local connection form. The holonomy of $g$ is contained in $\mathbb{R}\cdot \mathrm{diag}(1,0) \ltimes \mathfrak{m}(1,1)$ if and only if the following relations between the components are satisfied
\begin{align*}
b^1 (\nabla b_1) & =  b^2 (\nabla b_2)\qquad
b^1 (\nabla b_4)  = \sqrt{2} b^2 (\nabla b_7) \qquad 
b^1 (\nabla b_j)  =  0 \quad j = 2,7 \hspace{.5em}.
\end{align*}
These relations form a system of partial differential equations given by\\[\baselineskip]
\noindent $b^1 (\nabla b_1)  =  b^2 (\nabla b_2) :$ \hfill
\begin{align*}
2 (r_5)_{x_1} &= -  (v)_{x_6} + (r_6)_{x_2} 
\end{align*}
$b^1 (\nabla b_4)  = \sqrt{2} b^2 (\nabla b_7) :$ \hfill
\begin{align*}
\sqrt{2} (r_5)_{x_4} &=   (r_6)_{x_7} - (t_5)_{x_6}
\end{align*}
$b^1 (\nabla b_2)  =  0 : $ \hfill
\begin{align*}
 (r_5)_{x_2} &=  v \cdot (r_6)_{x_2} \\
  (r_6)_{x_2} &= -(u)_{x_6}
  \end{align*}
  $b^1 (\nabla b_7)  =  0 : $ \hfill
\begin{align*}
 (r_5)_{x_7} &= v \cdot (r_6)_{x_7}\\
 (r_6)_{x_7} &=- (t_5)_{x_6} \hspace{.5em}.
   \end{align*}
   Straightforward substitution and integration lead to expressions for $r_5,r_6$ as stated in the theorem, while $v,t_5$ and the functions $F_5,F_6$, which arise by integration, are arbitrary.
\end{proof}
\begin{example}
It is easy to verify that with the choice 
\begin{align*}
v(x_5,x_6) = \frac{1}{2}x_6^2 \qquad t_5(x_5,x_6) = \frac{1}{6} x_6^3 \qquad F_5,F_6 \; \mathrm{arbitrary.}
\end{align*}
the holonomy of $(M,g)$ is $\mathfrak{h} = \mathbb{R}\cdot \mathrm{diag}(1,0) \ltimes \mathfrak{m}(1,1)$.
\end{example}
\subsection{Holonomy of Type III 2b}
\begin{prop}
The holonomy of $(M^{4,3},g)$ is contained in $\mathrm{G}_2^{*}$ and is of Type III 2b if and only if there are local coordinates $x_1,\ldots,x_7$ such that $g = 2(b^1\cdot b^5 + b^2 \cdot b^6 + b^3\cdot b^7) - (b^4)^2$, where
\begin{align*}
b^{1} &= \mathrm{d}{x_1} +  r_5 (x_1,x_2,x_4,x_5,x_6,x_7)\cdot \mathrm{d}{x_5}+  r_6 (x_2,x_5,x_6,x_7) \cdot \mathrm{d}{x_6} \\
b^{6} &= \mathrm{d}{x_6} + v(x_5,x_6)\cdot \mathrm{d}{x_5} \\
b^{i} &= \mathrm{d}{x_i} \qquad \qquad i = 2,3,4,5,7 \hspace{.5em},
\end{align*}
with functions $r_5, r_6$ of the form
\begin{align*}
r_5 &=\frac{x_4}{\sqrt{2}} \cdot (F_6)_{x_7}- x_2 \cdot v \cdot (v)_{x_6} - x_1 \cdot (v)_{x_6} + F_5(x_5,x_6,x_7) \\
r_6&= -x_2 \cdot(v)_{x_6} + F_6 (x_5,x_6,x_7)
\end{align*}
for arbitrary functions $v,F_5,F_6$.
\end{prop}
\begin{proof}
The structure equation
\begin{align*}
\begin{pmatrix}
\mathrm{d}{b^{1}} \\ \mathrm{d}{b^{2}} \\ \mathrm{d}{b^{3}} \\ \mathrm{d}{b^{4}} \\ \mathrm{d}{b^{5}} \\ \mathrm{d}{b^{6}} \\ \mathrm{d}{b^{7}}
\end{pmatrix} =
-\begin{pmatrix}
\textbf{x} & 0 & 0 & \sqrt{2} \textbf{v} & 0 & -\textbf{y}_{1} & -\textbf{y}_{2} \\
0 & \textbf{x} & 0 & 0  & \textbf{y}_{1}& 0 & \textbf{v} \\
0 & 0 & 0 & 0  & \textbf{y}_{2}& - \textbf{v} & 0 \\
0 & 0 & 0 & 0 & \sqrt{2} \textbf{v}& 0 &0 \\
0 & 0 & 0 & 0 & -\textbf{x} & 0 &0 \\
0 & 0 & 0 & 0 & 0 & -\textbf{x} & 0 \\
0 & 0 & 0 & 0 & 0 & 0 & 0 \\
\end{pmatrix} \wedge 
\begin{pmatrix}
b^{1} \\ b^{2} \\ b^{3} \\ b^{4} \\ b^{5} \\ b^{6} \\ b^{7}
\end{pmatrix}
\end{align*}
 allows the choice of $b^7 = \mathrm{d}x_7$. Next,  the equations for $\mathrm{d}b^5$ and $\mathrm{d}b^6$ imply $b^5 = u \cdot \mathrm{d}x_5$ and $b^6 = v' \cdot \mathrm{d}x_6$. Transforming the basis with an element
\begin{align*}
\mathrm{GL}(2,\mathbb{R}) \ni M = \begin{pmatrix} u^{-1} & 0 \\ 0 & 1 \end{pmatrix} \cong \mathrm{diag}\left(u^{-1}, M, 1, u, (M^{\top})^{-1}\right) \in \mathrm{G}_2^{*}
\end{align*}
yields $b^5 = \mathrm{d}x_5$, $b^6 = \mathrm{d}x_6+ v \cdot \mathrm{d}x_5$ and $\mathbf{x} \in \mathcal{I}(\mathrm{d}x_5)$, since $\mathrm{d}b^5 = 0$. It follows $v = v(x_5,x_6)$.

From the structure equation deduce $\mathrm{d}b^4 \in \mathcal{I}(\mathrm{d}x_5)$, which is why the choice $b^4 = \mathrm{d}x_4 + q \cdot \mathrm{d}x_5$ is feasible. A change of basis with an appropriate element $\exp(h(0,v,0))$ leads to $b^4 = \mathrm{d}x_4$. Consequently, one has $\mathbf{v} \in \mathcal{I}(\mathrm{d}x_5)$, as $\mathrm{d}b^4 = 0$.

Next, the structure equation implies $\mathrm{d}b^2,\mathrm{d}b^3 \in \mathcal{I}(\mathrm{d}x_5)$, since $\mathbf{x},\mathbf{v} \in  \mathcal{I}(\mathrm{d}x_5)$ and $b^5 = \mathrm{d}x_5$. Then, 
\begin{align*}
b^2 = \mathrm{d}x_2 + s_5 \cdot \mathrm{d}x_5, \qquad b^3 = \mathrm{d}x_3 + t_5 \cdot \mathrm{d}x_5
\end{align*}
and after transforming with a suitable element $\exp(h(0,0,(y_1,y_2)^{\top}))$ obtain $b^2 = \mathrm{d}x_2$ and $b^3 =\mathrm{d}x_3$. This implicates $\mathbf{y}_1 \in \mathcal{I}(\mathrm{d}x_2,\mathrm{d}x_5,\mathrm{d}x_6)$ and $\mathbf{y}_2 \in \mathcal{I}(\mathrm{d}x_5,\mathrm{d}x_6)$.

Finally, as  $\mathbf{v} \in \mathcal{I}(\mathrm{d}x_5)$,  $\mathbf{y}_2 \in \mathcal{I}(\mathrm{d}x_5,\mathrm{d}x_6)$ and $b^6 = \mathrm{d}x_6 + v\cdot \mathrm{d}x_5$, one has $\mathrm{d}b^1 \in \mathcal{I}(\mathrm{d}x_5,\mathrm{d}x_6)$ by the structure equation. Thus,
\begin{align*}
b^1 = \mathrm{d}x_1 + r_5 \cdot \mathrm{d}x_5 + r_6 \cdot \mathrm{d}x_6
\end{align*}
and from the equation
\begin{align*}
\mathrm{d}{b^{1}} &= \mathrm{d}{r_5} \wedge \mathrm{d}{x_5} + \mathrm{d}{r_6} \wedge \mathrm{d}{x_6} =  - \textbf{x} \wedge b^{1} -\sqrt{2}\mathbf{v} \wedge b^{4} + \textbf{y}_1 \wedge b^{6} + \textbf{y}_2 \wedge b^{7} 
\end{align*}
obtain $r_5 = r_5(x_1,x_2,x_4,x_5,x_6,x_7)$ and $r_6 = r_6(x_2,x_5,x_6,x_7)$.

Consider the metric $g$ with respect to the local coordinates given as above. The holonomy of $g$ is contained in $\mathbb{R}\cdot \mathrm{diag}(1,0)\ltimes \mathfrak{m}(1,2)$ if and only if the following relations between the components of the local connection form hold:
\begin{align*}
b^1 (\nabla b_1) & =  b^2 (\nabla b_2),\qquad
b^1 (\nabla b_4)  = \sqrt{2} b^2 (\nabla b_7), \qquad
b^1 (\nabla b_2)  =  0 \hspace{.5em}.
\end{align*}
The system of partial differential equations associated with these relations is given by\\[\baselineskip]
\noindent $b^1 (\nabla b_1)  =  b^2 (\nabla b_2) :$ \hfill
\begin{align*}
2(r_5)_{x_1} &= - (v)_{x_6} + (r_6)_{x_2}
\end{align*}
$b^1 (\nabla b_4)  = \sqrt{2} b^2 (\nabla b_7) :$ \hfill
\begin{align*}
 \sqrt{2} (r_5)_{x_4} &= (r_6)_{x_7}
 \end{align*}
 $b^1 (\nabla b_2)  =  0 :$ \hfill
\begin{align*}
 - (v)_{x_6} &= (r_6)_{x_2}\\
  (r_5)_{x_2} &= v \cdot (r_6)_{x_2} \hspace{.5em}.
 \end{align*}
Again, the system can be solved straightforward by substitution and integration in order to obtain the form of the functions $r_5,r_6$ as stated in the theorem. The new functions $F_5,F_6$, arising due to integration, and the function $v$ are arbitrary.
\end{proof}
\begin{example}
By simple calculations one can check that the choice 
\begin{align*}
v(x_5,x_6 ) =\frac{1}{2} x_6^2 \qquad F_6(x_5,x_6,x_7) = -x_7^2 \qquad F_5 = 0
\end{align*}
leads to a metric with holonomy $\mathfrak{h} = \mathbb{R} \cdot \mathrm{diag}(1,0) \ltimes \mathfrak{m}(1,2)$.
\end{example}
\subsection{Holonomy of Type III 2c}
\begin{prop}
The holonomy of $(M^{4,3},g)$ is contained in $\mathrm{G}_2^{*}$ and is of Type III 2c if and only if there are local coordinates $x_1,\ldots,x_7$ such that $g = 2(b^1\cdot b^5 + b^2 \cdot b^6 + b^3\cdot b^7) - (b^4)^2$, where
\begin{align*}
b^{1} &= \mathrm{d}{x_1} + r_5 (x_4,x_5,x_6,x_7) \cdot \mathrm{d}{x_5} + r_7(x_5,x_6,x_7)\cdot \mathrm{d}{x_7}\\
b^{3} &= \mathrm{d}{x_3} + t_5(x_5,x_6)   \cdot  \mathrm{d}{x_5}\\
b^{i} &= \mathrm{d}{x_i} \qquad \qquad i = 2,4,5,6,7 \hspace{.5em}, 
\end{align*}
with functions $r_5,r_7$ given by 
\begin{align*}
r_5 &= - \sqrt{2} \cdot x_4 \cdot (t_5)_{x_6} + F_5 (x_5,x_6,x_7) \\
r_7 &=t_5 + F_7(x_5,x_7)
\end{align*}
for arbitrary functions $t_5,F_5,F_7$ which fulfil
\begin{align*}
(F_5)_{x_7} = (t_5)_{x_5} + (F_7)_{x_5} \hspace{.5em}.
\end{align*}
\end{prop}
\begin{proof}
Obviously, the structure equation
\begin{align*}
\begin{pmatrix}
\mathrm{d}{b^{1}} \\ \mathrm{d}{b^{2}} \\ \mathrm{d}{b^{3}} \\ \mathrm{d}{b^{4}} \\ \mathrm{d}{b^{5}} \\ \mathrm{d}{b^{6}} \\ \mathrm{d}{b^{7}}
\end{pmatrix} =
-\begin{pmatrix}
0 & 0 & 0 & \sqrt{2} \textbf{v} & 0 & -\textbf{y}_{1} & 0 \\
0 & 0 & 0 & 0  & \textbf{y}_{1}& 0 & \textbf{v} \\
0 & 0 & 0 & 0  & 0& - \textbf{v} & 0 \\
0 & 0 & 0 & 0 & \sqrt{2} \textbf{v}& 0 &0 \\
0 & 0 & 0 & 0 & 0 & 0 &0 \\
0 & 0 & 0 & 0 & 0 & 0 & 0 \\
0 & 0 & 0 & 0 & 0 & 0 & 0 \\
\end{pmatrix} \wedge 
\begin{pmatrix}
b^{1} \\ b^{2} \\ b^{3} \\ b^{4} \\ b^{5} \\ b^{6} \\ b^{7}
\end{pmatrix}
\end{align*}
 allows the choices $b^5 = \mathrm{d}x_5$, $b^6 = \mathrm{d}x_6$, $b^7 = \mathrm{d}x_7$. Furthermore, $\mathrm{d}b^4 \in \mathcal{I}(\mathrm{d}x_5)$ such that $b^4 = \mathrm{d}x_4 + q \cdot \mathrm{d}x_5$ can be chosen. A change of basis with an element $\exp(h(0,v,0))$ yields $b^4 = \mathrm{d}x_4$, $\mathrm{d}b^4 = 0$ and $\mathbf{v} \in \mathcal{I}(\mathrm{d}x_5)$. Using the latter, the structure equation implies $\mathrm{d}b^3 \in \mathcal{I}(\mathrm{d}x_5)$. Thus, choose $b^3 = \mathrm{d}x_3 + t_5 \cdot \mathrm{d}x_5$ and by comparing the expressions for $\mathrm{d}b^3$ we have $t_5 =t_5(x_5,x_6)$.

Next, from the structure equation deduce $\mathrm{d}b^2 \in \mathcal{I}(\mathrm{d}x_5)$ which leads to $b^2 = \mathrm{d}x_2 + s_5 \cdot \mathrm{d}x_5$. After a transformation with a suitable element $\exp(h(0,0,(y_1,0)^{\top}))$ obtain $b^2 = \mathrm{d}x_2$ and $\mathbf{y}_1 \in \mathcal{I}(\mathrm{d}x_5,\mathrm{d}x_7)$. 

Lastly, $\mathrm{d}b^1 \in \mathcal{I}(\mathrm{d}x_5,\mathrm{d}x_7)$ by the structure equation and $\mathbf{y}_1 \in \mathcal{I}(\mathrm{d}x_5,\mathrm{d}x_7)$. Thus,
\begin{align*}
b^1 = \mathrm{d}x_1+ r_5 \cdot \mathrm{d}x_5 + r_7 \cdot \mathrm{d}x_7
\end{align*}
and by comparing both sides of 
\begin{align*}
\mathrm{d}{b}^1 = \mathrm{d}{r}_5 \wedge \mathrm{d}{x_5} + \mathrm{d}{r}_7 \wedge \mathrm{d}{x}_7 = -\sqrt{2}\mathbf{v} \wedge b^4 + \mathbf{y}_1 \wedge b^6
\end{align*}
leads to $r_5 =r_5 (x_4,x_5,x_6,x_7)$ and $r_7 = r_7(x_5,x_6,x_7)$.

With respect to the local coordinates, the metric $g$ has holonomy contained in $\mathfrak{m}(1,1)$ if and only if the following relations between the components of the local connection form are satisfied:
\begin{align*}
 b^1 (\nabla b_7) &= 0, \qquad b^1 (\nabla b_4)  = \sqrt{2}\cdot b^2 (\nabla b_7) \hspace{.5em}.
\end{align*}
This is equivalent to the partial differential system\\[\baselineskip]
\noindent $b^1 (\nabla b_4) = \sqrt{2} \cdot b^2 (\nabla b_7) : $ \hfill
\begin{align*}
 \sqrt{2} \cdot (r_5)_{x_4} &= - (r_7)_{x_6} - (t_5)_{x_6} 
 \end{align*}
  $ b^1 (\nabla b_7) = 0 :$ \hfill
\begin{align*} 
(r_7)_{x_6} &= (t_5)_{x_6} \\
(r_7)_{x_5} &= (r_5)_{x_7}\hspace{.5em}.
\end{align*}
Solving the system by integration and substitution give rise to the dependencies of $r_5,r_6$ on $t_5$ and new functions $F_5,F_7$ as proposed in the proposition. The latter functions are arbitrary.
\end{proof}
\begin{example}
With the choice
\begin{align*}
t_5(x_5,x_6) = \frac{1}{2}x_6^2 \qquad F_5 = F_7 = 0
\end{align*}
it is simple to check that the holonomy of the corresponding metric $g$ is $\mathfrak{h} = \mathfrak{m}(1,1)$.
\end{example}
\subsection{Holonomy of Type III 2d}
\begin{prop}
The holonomy of $(M^{4,3},g)$ is contained in $\mathrm{G}_2^{*}$ and is of Type III 2d if and only if there are local coordinates $x_1,\ldots,x_7$ such that $g = 2(b^1\cdot b^5 + b^2 \cdot b^6 + b^3\cdot b^7) - (b^4)^2$, where
\begin{align*}
b^{1} &=  \mathrm{d}{x_1} + r_5 (x_4,x_5,x_6,x_7)\cdot \mathrm{d}{x_5}+  r_6 (x_5,x_6,x_7)\cdot \mathrm{d}{x_6}\\
b^{i} &= \mathrm{d}{x_i} \qquad \qquad i = 2,\ldots,7 \hspace{.5em},
\end{align*}
with $r_5$ is of the form
\begin{align*}
r_5 &=\frac{1}{ \sqrt{2}} \cdot x_4 \cdot (r_6)_{x_7}  + F(x_5,x_6,x_7)
\end{align*}
for arbitrary functions $r_6$ and $F$.
\end{prop}
\begin{proof}
The structure equation
\begin{align*}
\begin{pmatrix}
\mathrm{d}{b^{1}} \\ \mathrm{d}{b^{2}} \\ \mathrm{d}{b^{3}} \\ \mathrm{d}{b^{4}} \\ \mathrm{d}{b^{5}} \\ \mathrm{d}{b^{6}} \\ \mathrm{d}{b^{7}}
\end{pmatrix} =
-\begin{pmatrix}
0 & 0 & 0 & \sqrt{2} \textbf{v} & 0 & -\textbf{y}_{1} & -\textbf{y}_{2} \\
0 & 0 & 0 & 0  & \textbf{y}_{1}& 0 & \textbf{v} \\
0 & 0 & 0 & 0  & \textbf{y}_{2}& - \textbf{v} & 0 \\
0 & 0 & 0 & 0 & \sqrt{2} \textbf{v}& 0 &0 \\
0 & 0 & 0 & 0 & 0 & 0 &0 \\
0 & 0 & 0 & 0 & 0 & 0 & 0 \\
0 & 0 & 0 & 0 & 0 & 0 & 0 \\
\end{pmatrix} \wedge 
\begin{pmatrix}
b^{1} \\ b^{2} \\ b^{3} \\ b^{4} \\ b^{5} \\ b^{6} \\ b^{7}
\end{pmatrix}
\end{align*}
 allows the choice of $b^5 = \mathrm{d}x_5$, $b^6 = \mathrm{d}x_6$, $b^7 = \mathrm{d}x_7$. As a consequence one has $\mathrm{d}b^4 \in \mathcal{I}(\mathrm{d}x_5)$ such that one can choose $b^4 = \mathrm{d}x_4 +q  \cdot \mathrm{d}x_5$. A change of basis with a suitable element $\exp(h(0,v,0))$ leads to $b^4 = \mathrm{d}x_4$ and $\mathbf{v} \in \mathcal{I}(\mathrm{d}x_5)$, since $\mathrm{d}b^4 = 0$. 

Next, using $\mathbf{v} \in \mathcal{I}(\mathrm{d}x_5)$, the structure equation implies $\mathrm{d}b^2,\mathrm{d}b^3 \in \mathcal{I}(\mathrm{d}x_5)$ such that the assumption $b^2 = \mathrm{d}x_2 + s_5 \cdot \mathrm{d}x_5$ and $b^3 = \mathrm{d}x_3 + t_5 \cdot \mathrm{d}x_5$ is possible. After a transformation using an appropriate element $\exp(h(0,0,(y_1,y_2)^{\top}))$ obtain $b^2 = \mathrm{d}x_2$ and $b^3 = \mathrm{d}x_3$. Consequently, $\mathbf{y}_1 \in \mathcal{I}(\mathrm{d}x_5,\mathrm{d}x_7)$ and $\mathbf{y}_2 \in \mathcal{I}(\mathrm{d}x_5,\mathrm{d}x_7)$ by the structure equation.

Finally, this implies $\mathrm{d}b^1 \in \mathcal{I}(\mathrm{d}x_5,\mathrm{d}x_6)$, since $\mathbf{v} \in \mathcal{I}(\mathrm{d}x_5)$ and $b^6 =\mathrm{d}x_6$. Thus,
\begin{align*}
b^1 = \mathrm{d}x_1 + r_5 \cdot \mathrm{d}x_5 + r_6 \cdot \mathrm{d}x_6
\end{align*}
and comparing both sides of
\begin{align*}
\mathrm{d}{b^{1}} &= \mathrm{d}{r_5} \wedge \mathrm{d}{x_5}+ \mathrm{d}{r_6} \wedge \mathrm{d}{x_6}  =  -\sqrt{2}\mathbf{v} \wedge b^{4} + \mathbf{y}_1 \wedge b^{6} + \mathbf{y}_2 \wedge b^{7}
\end{align*}
yields $r_5 =r_5(x_4,x_5,x_6,x_7)$ and $r_6 = r_6(x_5,x_6,x_7)$.

A metric $g$ as given by the proposition has, with respect to the local coordinates, holonomy contained in $\mathfrak{m}(1,2)$ if and only if the relation
\begin{align*}
b^1 (\nabla b_4) & = \sqrt{2} b^2 (\nabla b_7)
\end{align*}
between the components of the local connection form is satisfied. This results in the partial differential equation 
\begin{align*}
\sqrt{2} (r_5)_{x_4} = (r_6)_{x_7} 
\end{align*}
which is, by integration, equivalent to the stated dependence of $r_5$ on $r_6$ and a new function $F$. 
\end{proof}
\begin{example}
It is easy to check that the metric $g$ associated with the choice
\begin{align*}
r_6(x_5,x_6,x_7) = -2\cdot x_6 x_7 \qquad F(x_5,x_6,x_7) =\frac{1}{2} x_7^2 
\end{align*}
has holonomy $\mathfrak{h} = \mathfrak{m}(1,2)$.
\end{example}
\begin{remark}
The algebras of Type III 2c and 2d are abelian and of dimension two or three. In \cite{Kath2013} indefinite symmetric spaces were classified and it turned out that their holonomy algebras are also abelian and of dimension two or three. They indeed coincide, which can be seen by an appropriate change in notation. Thus, the algebras of Type III 2c and 2d are holonomy algebras of (locally) symmetric spaces.
\end{remark}
\bibliographystyle{amsplain}
\bibliography{Bibliography_TypeIII}   
\end{document}